\documentclass{amsart}
\usepackage{amsmath,amssymb}
\usepackage[mathscr]{eucal}
\usepackage[all]{xy}

\theoremstyle{plain}
\newtheorem{theorem}{Theorem}[section]
\newtheorem{proposition}[theorem]{Proposition}
\newtheorem{lemma}[theorem]{Lemma}
\newtheorem{corollary}[theorem]{Corollary}

\newtheorem{theorem*}{Theorem}[]

\theoremstyle{definition}

\newtheorem{example}[theorem]{Example}

\theoremstyle{remark}
\newtheorem{remark}[theorem]{Remark}

\newcommand{\secref}[1]{Section~\ref{#1}}
\newcommand{\thmref}[1]{Theorem~\ref{#1}}
\newcommand{\propref}[1]{Proposition~\ref{#1}}

\newcommand{\corref}[1]{Corollary~\ref{#1}}

\newcommand{\exref}[1]{Example~\ref{#1}}

\newcommand{\remref}[1]{Remark~\ref{#1}}

\def\lbr{{\big[\![}}
\def\rbr{{]\!\big]}}
\def\:{{\colon}}
\def\WL{\mathrm{WL}}

\def\Q{{\mathbb Q}}

\def\L{{\mathbb L}}
\def\map{\mathrm{map}}
\def\Cl{\mathrm{cl}}

\def\Der{\mathrm{Der}}

\def\Hnil{\mathrm{Hnil}}
\def\Rel{\mathrm{Rel}}
\def\cat0{\mathrm{cat}_0}

\def\ad{\mathrm{ad}}
\def\induced{(\zeta_{a} \mid \zeta_{b})_{\!  \psi}}
\def\inducedddd{(\zeta_{a} \mid \zeta_{b} \mid \zeta_c)_{\!  \psi}}
\def\inducedd{(\zeta_{a} \mid \zeta_{b})_{\!  \psi *}}
\def\induceddd{(\zeta^*_{a} \mid \zeta^*_{b})_{\! \mathcal{L}_f}}

\def\cat{{\mathrm {cat}}}

\def\QL{\mbox{\boldmath$\mathscr{L}$} }

\begin{document}

\title[Whitehead Products in Function Spaces]
{Whitehead Products in Function Spaces: Quillen Model Formulae}
\author{Gregory  Lupton}
\address{Department of Mathematics,
      Cleveland State University,
      Cleveland OH 44115 U.S.A.}
\email{G.Lupton@csuohio.edu}
\author{Samuel Bruce Smith}
\address{Department of Mathematics,
  Saint Joseph's University,
  Philadelphia, PA 19131 U.S.A.}
\email{smith@sju.edu}
\date{\today}
\keywords{Whitehead product, Function Space, Quillen minimal model,
derivation, coformal space, Whitehead length}
\subjclass[2000]{55P62, 55Q15}

\begin{abstract}  We study
Whitehead products in the rational homotopy groups of a general
component of a function space.  For the component of any based map
$f \colon X \to Y$, in either the based or free function space, our
main results express the Whitehead product directly in terms of the
Quillen minimal model of $f$.  These results follow from a purely
algebraic development in the setting of chain complexes of
derivations of differential graded Lie algebras, which is of
interest in its own right.  We apply the results to study the
Whitehead length of function space components.
\end{abstract}

\maketitle

\section{Introduction}
Let $f \colon X \to Y$ be a based map of based, simply connected CW
complexes with $X$ a finite complex. Let $\map(X, Y;f)$ denote the
path component containing $f$  in the space of basepoint-free
continuous functions from  $X$ to $Y,$ and $\map_{*}(X, Y;f)$ the
component in the space of basepoint-preserving functions. In this
paper, we study the structure of the Whitehead product on the
rational homotopy groups  of these function spaces.

The paper is organized as follows.  In \secref{sec:Gamma}, we
describe the Quillen model of the map
$$ \eta \times 1 \colon S^{p+q-1} \times X \to  (S^{p}\vee S^{q}) \times X $$
where $\eta$ is the Whitehead product. Our description is given in
the framework of chain complexes of generalized derivations of
Quillen models, which was introduced  in \cite{L-S2} in order to
identify the rational homotopy groups of function space components.
\secref{sec:algebra} is a purely algebraic development in the
setting of chain complexes that arise in the category of
differential graded (DG) Lie algebras. Using the form of the Quillen
model of $\eta \times 1$ as a guide, we construct a ``Whitehead
product" on the homology of the mapping cone of a map of DG Lie
algebras.  We extend our construction to the chain complexes of
generalized derivations mentioned above. In \secref{sec:iterated},
we record a detailed formula useful for applications, and mention
briefly some extensions, such as iterated products. In
\secref{sec:formula}, we return to the topological setting and prove
our main result:  we identify Whitehead products in the rational
homotopy groups of $\map(X, Y;f)$ and $\map_{*}(X, Y;f)$ with the
``Whitehead products" constructed algebraically from the Quillen
model of the map.

We present various applications in \secref{sec:applications}, where
we study the rational Whitehead length of function space components.
Given a space $Z,$ let $\WL(Z)$, the \emph{Whitehead length of $Z$},
denote the length of longest, non-zero iterated Whitehead bracket in
$\pi_{\geq2}(Z).$ (We avoid considerations of the fundamental group
throughout this paper.) Thus  $\WL(Z) = 1$ means all Whitehead
products vanish and $\WL(Z) \geq 2$ means that there exists a
non-trivial Whitehead product.  Let $\WL_\Q(Z)$, the \emph{rational
Whitehead length of $Z$}, denote the length of longest, non-zero
iterated Whitehead bracket in $\pi_{\geq2}(Z) \otimes \Q$. We first
observe that, for the null component of a function space, we have
$\WL_\Q(\map(X, Y; 0)) = \WL_\Q(Y)$ as a consequence of classical
ideas (\thmref{thm:WL null component}). Using our formula, we then
prove that, for any map $f \colon X \to Y,$ that is, for a general
component, we have
$$\max\{\WL_\Q(\map_*(X, Y;f)), \WL_\Q(\map(X, Y;f)) \} \leq \WL_\Q(Y)$$
provided $Y$ is a coformal space (\thmref{thm:Y coformal}). Focusing
on the based function space, we also prove that  $$ \WL_\Q(\map_*(X,
Y;f)) \leq \Cl_0(X),$$ where $\Cl_0(X)$ denotes the rational
cone-length of $X$ (\thmref{thm:cone}) complementing the
corresponding (integral) result at the null component due to Ganea
\cite{Gan60}. In \thmref{thm:sphere}, we  apply our formulae to give
a complete calculation of the rational Whitehead length of all
components of $\map(X, S^n)$ and of $\map_*(X, S^n)$ for $X$ a
finite, simply connected CW complex. Finally, we show that the
inequality
$$\WL_\Q(\map(X, Y;f)) > \WL_\Q(\map(X, Y; 0)) =  \WL_\Q(Y)$$ may hold.
Precisely, in \exref{ex:nonabelian}, we give a space $Y$ with
vanishing rational Whitehead products and a map $f \colon S^3 \to Y$
such that $\WL_\Q(\map(S^3, Y; f)) \geq 2.$

We assume familiarity with rational homotopy theory from Quillen's
point of view.  Our main reference for this material is \cite{F-H-T}
(see also \cite{Q, Tan}).  We introduce notation as we go but recall
here that a map $f \colon X \to Y$ of simply connected CW complexes
of finite type has a {\em Quillen minimal model} which is a map
$\mathcal{L}_f \colon (\mathcal{L}_X, d_X) \to (\mathcal{L}_Y, d_Y)$
of connected DG Lie algebras over $\Q$. The Quillen minimal model of
$f$ is a complete invariant of the rationalization of $f$. In
particular, there is a  natural isomorphism $H_*(\mathcal{L}_X, d_X)
\cong \pi_*(\Omega X) \otimes \Q$ of graded Lie algebras. The map
induced by $f$ on rational homotopy Lie algebras corresponds, with
these identifications for $X$ and $Y,$ to  the map induced by
$\mathcal{L}_f$ on homology. Our main results explain how the
Whitehead product in the rational homotopy groups of $\map(X,Y;f)$
and $\map_*(X, Y;f)$ depends on $\mathcal{L}_f.$

\begin{remark}
Rational Whitehead products for  function spaces have been studied
by several authors. In \cite{V}, Vigu\'{e}-Poirrier gave an elegant
formula for Whitehead products in the null-components $\map_*(X,
Y;0)$ and $\map(X, Y;0)$ (including degree $1$) directly in terms of
products in the rational homotopy of $Y$ and the cup product in
$H^*(X, \Q)$ under certain restrictions on $X$ and $Y$. This result
was recently extended to full generality by  Buijs and  Murillo as a
special case of their description of the rational homotopy Lie
algebra of any component of a function space \cite{B-M}. Also, we
mention the recent work of  Buijs, F\'{e}lix and  Murillo \cite{BFM}
which identifies a Lie model for spaces of sections and, in
particular, for components of a function space.

Our work differs from these other results in at least two respects.
First, we describe rational Whitehead products for general function
space components by means of a construction that proceeds directly
from the Quillen model of a map.  Because we focus on a description
specifically at the level of rational homotopy groups, rather than a
more comprehensive description of the rational homotopy type, we are
able to give a fairly direct construction: our description lends
itself well to the study of specific examples. Second, our
construction of topological (rational) Whitehead products is
developed from a purely algebraic one on the mapping cone of certain
maps of chain complexes (see \secref{sec:algebra}). This provides
the basis for new developments either in the algebraic settings, or
in topological situations other than function spaces that correspond
to mapping cones.
\end{remark}

\noindent {\bf Acknowledgement.} We are indebted to Yves F\'{e}lix
for many helpful discussions, and to the Universit\'{e} Catholique
de Louvain for hospitality, during the early stages of this project.
We thank the referee for a very careful reading of the paper.

\section{The Quillen Model of a certain map }
\label{sec:Gamma} We review the development of ideas in \cite{L-S2}.
An element $\alpha \in \pi_p(\map(X, Y;f))$ is represented by a map
$a \colon S^p \to \map(X, Y;f)$ whose adjoint is a map $A \colon S^p
\times X \to Y$ that restricts to $f \colon X \to Y$ on $X.$ By
considering the Quillen minimal model of the adjoint $A$ we are led
to consider a certain complex of (generalized) derivations of
Quillen models, which we denoted by $\Der(\mathcal{L}_X,
\mathcal{L}_Y; \mathcal{L}_f)$ in \cite{L-S2}. The homology groups
of this complex may be identified with the homotopy groups of the
based mapping space $\map_*(X, Y; f),$ and the homology groups of
the mapping cone of the (generalized) adjoint map
$$ \ad_{\mathcal{L}_f} \colon \mathcal{L}_Y \to \Der(\mathcal{L}_X, \mathcal{L}_Y; \mathcal{L}_f)$$
may be identified with the homotopy groups of  $\map(X, Y; f)$ (see
\cite[Th.3.1]{L-S2}). It is in this context that we wish to describe
the Whitehead product.

Topologically, a Whitehead product $\gamma = [\alpha, \beta]_w \in
\pi_{p+q-1}(\map(X, Y; f))$, for $\alpha \in \pi_{p}(\map(X, Y; f))$
and $\beta \in \pi_{q}(\map(X, Y; f))$, is represented by  the
composition
$$\xymatrix{S^{p+q-1} \ar@/_2pc/[rr]_{\gamma}\ar[r]^-{\eta} & S^p \vee S^q
\ar[r]^-{(a \mid b)} & \map(X,Y;f)},$$
where $\eta = [\iota_1, \iota_2]_w$ is the ``universal example" of a
Whitehead product.  The adjoint $C$ of $\gamma$ is the composition
$$\xymatrix{S^{p+q-1}\times X \ar@/_2pc/[rr]_{C} \ar[r]^-{\eta \times 1} & (S^p \vee
S^q)\times X \ar[r]^-{(A\mid
B)_f} & Y}.$$
As in the previous paragraph, we will translate this adjoint into the setting of
complexes of (generalized) derivations of Quillen models.  In order to do so,
a description of the Quillen model of $\eta \times 1$ is germane.

We say a  graded rational vector space $(V,d)$ with a differential $d$ of
degree $-1$ is a {\em DG space} or, alternately, a chain complex.
 By a {\em DG Lie algebra} $(L, d)$ we will
mean a connected, graded Lie algebra $L$ with
 bilinear product $[ \, , \, ]$ satisfying
\begin{itemize}
\item[(a)]
$|[x, y]| = |x | + |y |$
\item[(b)] $[x, y] = (-1)^{|x||y| + 1}[y, x] \hbox{\
\ and} $
\item[(c)] $[x, [y, z]] = [[x,y], z] +
 (-1)^{|x||y|}[y, [x, z]]$
\end{itemize}
and differential satisfying  $$d([x, y]) =
[d(x),
y] + (-1)^{|x|}[x, d(y)].$$
We write $\L(V)$ for the free graded Lie algebra generated by the
graded space $V$ and extend this notation, writing $\L(V,W)$ for the
free Lie algebra generated by $V$ and $W$ and  $\L(V, a)$ for the free
Lie algebra generated by $a$ and $V$ where $a$ is an element of
homogeneous degree. We write  $\L(V; d)$ for the DG Lie algebra
$(\L(V), d)$.

We recall that a DG Lie algebra $(L, d)$ has an associated DG Lie
algebra of derivations $(\Der(L), D)$. Here $\Der(L)$ denotes the
graded space of positive-degree derivations of $L$ with the usual
graded commutator product of derivations, that is,
$$[\theta, \phi] = \theta\circ\phi - (-1)^{|\theta| |\phi|}\phi\circ\theta$$
for $\theta, \phi \in \Der(L)$, and differential $D(\theta) = [d,
\theta] = d\theta - (-1)^{|\theta|} \theta d$.  Then the adjoint
$\ad \colon (L, d) \to (\Der(L), D)$, defined by $\ad(l)(l') =
[l,l']$, is a map of DG Lie algebras. We are interested in a natural
generalization of this set-up.  Let $\psi \colon (L, d_{L}) \to (K,
d_{K})$ be a given DG Lie algebra map. Define a {\em
$\psi$-derivation} of degree $n$ to be a linear map $\theta \colon
L_{*} \to K_{*+n}$ satisfying
$$ \theta([x, y]) = [\theta(x),
\psi(y)] + (-1)^{n|x|}[\psi(x), \theta(y)].$$ We write $\Der_{n}(L,
K; \psi)$ for the space of degree-$n$ $\psi$-derivations. The
differential $D_{\psi}$ defined by
$$D_{\psi}(\theta) = d_{K} \circ \theta - (-1)^{|\theta|}\theta \circ
d_{L}$$ makes the pair $(\Der(L, K; \psi), D_{\psi})$ a DG space.
The $\psi$-adjoint (or ``generalized adjoint'') map
$$\ad_{\psi} \colon (K,d_{K}) \to (\Der(L, K; \psi), D_{\psi}),$$
given by $\ad_{\psi}(\alpha)(x) = [\alpha, \psi(x)]$ for $x \in L,
\alpha \in K,$  is a map of DG spaces.

Our description of the Quillen model of $\eta \times 1$ requires  a
construction featuring these generalized derivations.  Let $L =
\L(V; d)$ be a free DG Lie algebra. Let $p_{1}, \ldots, p_{n}$ be
given   integers $> 1$ and $a_{1}, \ldots, a_{n}$ elements of degree
$p_{1}-1, \ldots, p_{n}-1.$  Write $V^{a_{i}} = s^{p_{i}}(V)$ for
the $p_{i}$th suspension of $V$ and let $S_{a_{i}} \colon V \to
V^{a_{i}}$ denote the corresponding degree $p_{i}$ linear map.   We
define a new DG Lie algebra $(L(a_{1},\ldots, a_{n}), \partial)$ by
setting
\begin{equation} \label{eq:L(a,b)} L(a_{1}, \ldots, a_{n}) =
\L(V,a_{1}, \ldots, a_{n},  V^{a_{1}}, \ldots, V^{a_{n}}).
\end{equation}
Observe that the suspension $S_{a_{i}} \colon V \to V^{a_{i}}$
extends as a derivation to an element $S_{a_{i}} \in \Der_{p_{i}}(L,
L(a_{1},\ldots, a_{n}); \lambda)$ where $\lambda \colon L \to
L(a_{1},\ldots,a_{n})$ is the inclusion. Using this, we define the
differential as follows:
$$\partial(v) = d(v), \partial(a_{i}) =
0 \hbox{\, and \,}
\partial\left(S_{a_{i}}(v)\right) = (-1)^{p_{i}-1}[a_{i}, v] + (-1)^{p_i}S_{a_i}(dv)$$
for $v \in V.$  The definition of $\partial$ gives  the boundary relation
\begin{equation}\label{eq:Sa} D_{\lambda}\left(S_{a_{i}}\right) =
(-1)^{p_i-1}\ad_{\lambda}(a_{i}) \in \Der(L, L(a_{1},\ldots, a_{n});
\lambda).
\end{equation}

Recall that a simply connected CW complex $X$ of finite type admits
a Quillen minimal model $\mathcal{L}_{X} = \L(V; d_{X})$ which is a
free minimal DG Lie algebra with $V \cong s^{-1}\widetilde{H}_{*}(X;
\Q)$ and $H_{*}(\mathcal{L}_{X}) \cong \pi_{*}(\Omega X) \otimes
\Q.$ A map $f \colon X \to Y$ between such spaces induces a DG Lie
algebra map $$\mathcal{L}_{f} \colon (\mathcal{L}_{X}, d_{X}) \to
(\mathcal{L}_{Y}, d_{Y}).$$ The connection to   the map  $\eta
\times 1 \colon S^{p+q-1} \times X \to (S^{p} \vee S^{q}) \times X$
is provided by the following result.

\begin{theorem} \label{thm:SxX}\emph{\cite[Th.2.1]{L-S2}} Let $X$ be a simply
connected CW complex of finite type.  The DG Lie algebra
$\left(\mathcal{L}_{X}(a_{1}, \ldots, a_{n}), \partial \right)$
defined by {\em (\ref{eq:L(a,b)}) }is the
Quillen minimal model for the space
$\left(\vee_{i=1}^{n} S^{p_{i}}\right) \times X.$ \qed
\end{theorem}

By \thmref{thm:SxX}, the Quillen  model for $\eta \times 1$ is some
map of DG Lie algebras
$$ \Gamma \colon (\mathcal{L}_{X}(c), \partial_{c}) \to
(\mathcal{L}_{X}(a,b), \partial_{a,b})$$
where $|a| = p-1, |b|=q-1$ and $|c| = p+q-2.$
It is easy to check that $\Gamma(\chi) = \chi$ for $\chi \in
\mathcal{L}_{X}$ while $\Gamma(c) = (-1)^{p-1}[a,b].$ (This sign is appropriate per the
identifications of \cite[Ch.X.7.10]{GW}.) Let us write $S_{[a,b]}$
for the degree $p+q-1$ linear map  induced by $\Gamma$ via the rule
$$S_{[a,b]}(v) =_{\mathrm{def}}\Gamma(S_{c}(v))$$
for $v \in V.$  Then $S_{[a,b]}$ extends to a $\lambda$-derivation
$S_{[a,b]} \in \Der_{p+q-1}(\mathcal{L}_{X},\mathcal{L}_{X}(a,b); \lambda)$
satisfying the boundary
relation
\begin{equation} \label{eq:boundary} D_{\lambda}\left(S_{[a,b]}\right) =
(-1)^{q-1}\ad_{\lambda}([a,b]) \in
\Der_{p+q-2}(\mathcal{L}_{X}, \mathcal{L}_{X}(a,b); \lambda).
\end{equation}
Working backward, we see the identification of the derivation $S_{[a,b]}$
satisfying (\ref{eq:boundary}) completely
determines the Quillen minimal model of $\Gamma.$

In the next section, we find a  formula for $S_{[a,b]}$. In fact, we
identify this derivation as the ``universal example" for Whitehead
products constructed in the category of DG Lie algebras. (See Remark
\ref{rem:S[a,b]}, below.) To explain this further, we introduce the
{\em mapping cone} of a DG vector space map $\psi \colon (V, d_V)
\to (W, d_W)$, which we denote by $(\Rel(\psi), \delta_\psi)$. This
is the DG space  with $\Rel_n(\psi) =  V_{n-1} \oplus W_{n}$ and
differential $\delta_{\psi}$   defined as $\delta_\psi(v, w) =
(-d_V(v), \psi(v) + d_W(w))$.   The construction yields a short
exact sequence of DG spaces $  (W, d_{W}) \to (\Rel(\psi),
\delta_{\psi}) \to (V, d_V)$ giving rise to a long exact homology
sequence whose connecting homomorphism is $H(\psi)$. Applying this
to the adjoint $\ad_{\lambda} \colon (L(a,b),d) \to (\Der(L, L(a,b);
\lambda), D_{\lambda})$ we see the boundary conditions (\ref{eq:Sa})
and (\ref{eq:boundary}) are equivalent to the elements
$$ \zeta_{a} = \left((-1)^{p}a, S_{a}\right), \ \
\zeta_{b} = \left((-1)^{q}b, S_{b}\right) \hbox{\ \ and \ \ }
\zeta_{[a,b]} = \left((-1)^{q}[a,b], S_{[a,b]}\right)$$ being three
$D_{\lambda}$-cycles in $\Rel(\ad_\lambda)$ of degree $p, q$ and
$p+q-1,$ respectively.  In the next section, we construct a
Whitehead product $[ \, , \, ]_{w}$ on $H_{*}(\Rel(\ad_\lambda))$
satisfying
$$    \left[ \langle \zeta_{a} \rangle,
\langle \zeta_{b} \rangle \right]_{w} = \langle \zeta_{[a,b]}
\rangle$$ thereby completing the description of $\Gamma$, the
Quillen model of $\eta \times 1,$ above.

\section{Whitehead products in the category of DG Lie Algebras}
\label{sec:algebra} In this section, we describe  the construction
of  Whitehead products on the homology of chain complexes of
derivations arising from a given DG Lie algebra map $\psi \colon (L,
d_{L}) \to (K, d_{K})$.  We will approach our final construction in
several steps. First we give the definition of a Whitehead product,
referring to the classical correspondence between Whitehead products
and Samelson products.   Let $sL$ denote the suspension of $L$.
Given $x, y \in L$ define a bilinear pairing on $sL$ by the rule
$$[sx, sy]_{w} =_{\mathrm{def}}  (-1)^{|x|} s[x, y].$$
The   pairing $[ \, , \, ]_{w}$ then satisfies the identities
\begin{itemize}
\item[(i)] $ \left|[\alpha,\beta]_w \right| = |\alpha | + |\beta| - 1$
\item[(ii)] $[\alpha, \beta ]_w = (-1)^{|\alpha||\beta|} [ \beta, \alpha
]_w \hbox{\, and}
 $
 \item[(iii)] $ [ \alpha,  [\beta  , \gamma]_w]_w  =
 (-1)^{|\alpha|+1} [[ \alpha,  \beta ]_w , \gamma]_w +
 (-1)^{(|\alpha|+1)(|\beta|+1)} [\beta,  [\alpha  , \gamma]_w]_w  $
 \end{itemize}
 for $\alpha,\beta,\gamma \in sL.$
These identities    correspond, of course, to those satisfied by the higher  homotopy groups
of a space with the Whitehead product \cite[Chapter X.7]{GW}.
We denote a bilinear pairing satisfying (i)-(iii) by $[ \, , \, ]_{w}$
and call it a  {\em
 Whitehead product}.

As a preliminary, we next observe that a kind of ``pre-Whitehead product'' may be defined on any DG Lie algebra $(L, d_{L})$.
Specifically, define a   bilinear pairing on
$L$ by setting
\begin{equation} \label{eq:L product}\{ x , y\} =_{\mathrm{def}}  (-1)^{|x|+1}[x,
d_{L}(y)]. \end{equation}
The pairing $\{ \, , \, \}$ clearly satisfies (i).  Further,  we have the
following:
\begin{proposition}\label{prop:Jacobi}  The bilinear pairing $\{
\, , \, \}$ defined   on $L$ by {\em (\ref{eq:L product})} satisfies the identities  $(ii)$ and $(iii)$
up to boundaries in $(L, d_{L}).$
\end{proposition}
\begin{proof}
Write $\sim$ for the homologous relation  in $(L, d_{L})$ and let $d =
d_{L}.$  Let $p =
|x|, q= |y|$ and $r = |z|.$
Use the boundary $$d\left([x, y]\right) = [d(x), y] +
(-1)^p[x, d(y)]$$ to obtain
$$ \begin{array}{lllll}
\{ x, y \}   & = (-1)^{p+1}[x, d(y) ]
           & \sim [d(x), y ]
          & = (-1)^{(p-1)q + 1} [y, d(x)]
           & = (-1)^{pq} \{ y, x \}

\end{array}
$$
For (iii), observe
$$\begin{array}{ll}
\left\{ x, \{ y , z \}\right\} &= (-1)^{p+q}
 \left[ x, d\left([y, d(z) ]\right) \right]
 \\
&=
 (-1)^{p+q}
 \left[ x, [d(y), d(z) ] \right]
 \\   &= (-1)^{p+q}\left[[x, d(y)], d(z)\right]
 + (-1)^{q(p+1)}\left[ d(y),[x, d(z)]\right] \\
\end{array}$$
Then note that
$$  (-1)^{p+q}\left[[x, d(y)], d(z)\right] =
(-1)^{q+1}\left[ \{ x, y \}, d(z) \right] =
(-1)^{p+1}\left\{ \{x, y\}, z \right\}.$$
Finally,   the boundary
$$d\left(\left[ y, [ x,
d(z)] \right] \right)
= \left[d(y),  [x,
d(z) ] \right] + (-1)^q \left[y, d\left([x,
d(z) ]\right) \right]$$
implies $$
\begin{array}{ll}
(-1)^{q(p+1)}\left[ d(y),[x, d(z)]\right] & \sim (-1)^{pq+1}\left[
y, d \left([x, d(z) ]\right) \right]  \\  &= (-1)^{(p+1)(q+1)}
\left\{y, \{x, z\} \right\}. \qed\end{array} $$
\renewcommand{\qed}{}\end{proof}

Next  we consider  the case of a DG Lie algebra map $\psi \colon
(L, d_{L}) \to (K, d_{K})$  and its mapping cone $(\Rel(\psi),
\delta_{\psi}).$ We will construct a Whitehead product on the
homology of $(\Rel(\psi), \delta_\psi)$. Notice that this is a chain
complex, not a DG Lie algebra; it is not immediately evident that
such a product may be defined. Our construction here refers to the
two previous steps.

Let $(a, \alpha) \in \Rel_{p}(\psi)$ and  $(b, \beta) \in
\Rel_{q}(\psi)$ be given. Recall that this means $a \in L_{p-1}, b
\in L_{q-1}$ while $\alpha \in K_{p}, \beta  \in K_{q}$. Define  a
bilinear pairing  $\lbr \, , \, \rbr$, using the ordinary bracket in
$L$ but the pairing defined by (\ref{eq:L product}) in $K$, by
setting
\begin{equation} \label{eq:Wh Rel} \lbr (a, \alpha), (b, \beta) \rbr =_{\mathrm{def}}
\left( (-1)^{p}[a, b], \{\alpha, \beta \} \right) = \left( (-1)^{p}[a, b],(-1)^{p+1}[\alpha, d_K(\beta)] \right).
\end{equation}
We then have the following:
\begin{proposition} \label{prop:Wh on Rel}
    Let $\psi \colon (L, d_{L})
    \to (K, d _{K})$ be a DG Lie algebra map with mapping cone
 $(\Rel(\psi), \delta_{\psi}).$
The bilinear pairing  $\lbr \, , \, \rbr$ on
$\Rel(\psi)$ defined by {\em (\ref{eq:Wh Rel})} induces a
Whitehead product $[ \, , \, ]_{w}$ on  $H_{*}(\Rel(\psi)).$
\end{proposition}
\begin{proof}
 For suppose that $(a, \alpha)$ and $(b, \beta)$ are
$\delta_{\psi}$-cycles. Then $d_{L}(a) = d_{L}(b) = 0$ while $d_{K}(\alpha) =-
\psi(a)$ and $d_{K}(\beta) = -\psi(b).$ Observe that
$$d_{K}\left(\{\alpha, \beta\}\right)
=(-1)^{p+1}d_{K}\left(\left[\alpha, d_{K}(\beta) \right] \right)
= (-1)^{p+1}\left[d_{K}(\alpha), d_{K}(\beta)\right] =
-\psi\left((-1)^{p}[a, b]\right).$$
Thus the product $\lbr (a, \alpha), (b, \beta) \rbr$ is a
$\delta_{\psi}$-cycle, as well.

Next suppose
$(a, \alpha)  = \delta_{\psi}\left(c, \gamma \right)$ is a
$\delta_{\psi}$-boundary  and $(b,
\beta)$ is again a $\delta_{\psi}$-cycle. Then
$$ \lbr (a, \alpha), (b, \beta) \rbr = \delta_{\psi}\left((-1)^{p}[c, b], -\{\gamma, \beta \} \right)
$$
is a $\delta_{\psi}$-boundary, as well. To verify this in the second
variable observe that
$$ \begin{array}{ll}
d_{K} \left(- \{ \gamma, \beta \}\right) + (-1)^{p}\psi([c, b]) &=
(-1)^{p+1}[d_{K}(\gamma),d_{K}(\beta)]
+ (-1)^{p}\psi([c, b]) \\
&= (-1)^{p+1}\left[ \alpha - \psi(c), d_{K}(\beta)\right] + (-1)^{p}\psi([c,
b]) \\
& = \{ \alpha, \beta\},
\end{array}
$$
since $d_{K}(\beta) = - \psi(b).$

The pairing $\lbr \, , \, \rbr$ thus induces a bilinear pairing $[
\, , \, ]_{w}$   on
$H_*(\Rel(\psi))$
satisfying the degree condition (i)
 by construction.  The induced product satisfies  (ii) and (iii)
in the second variable  by \propref{prop:Jacobi}. In the first
variable, $[ \, , \, ]_{w}$  corresponds to the classical Whitehead
product (except with grading reduced one instead of increased one).
\end{proof}

Our final step is to consider  the (generalized) adjoint
$$\ad_{\psi} \colon (K, d_{K}) \to  \left(\Der(L, K; \psi),
D_{\psi}\right)$$ and its mapping cone  $\left( \Rel(\ad_{\psi}),
\delta_{\ad_{\psi}} \right).$ We define Whitehead products on the
homology of the latter two complexes. Notice that, once again,
neither of these complexes is a DG Lie algebra.

As at the start of \secref{sec:Gamma}, the Whitehead product of two
elements $\alpha \in \pi_p(X)$ and $\beta \in \pi_q(X)$ involves the
``universal example'' of such Whitehead products, namely $\eta \in
\pi_{p+q-1}(S^p \vee S^q)$. This is then mapped into
$\pi_{p+q-1}(X)$ by $ (\alpha \! \mid \! \beta),$ a map induced by
the given homotopy elements. Here, we take a similar approach. Given
two elements of $H(\Rel(\ad_\psi)),$ of degree $p$ and $q$ we first
describe a  ``universal example'' of the Whitehead product in
$H_{p+q-1}(\Rel(\ad_\lambda))$ (see (\ref{eq:universal example})
below). This is then mapped to $H_{p+q-1}(\Rel(\ad_\psi))$ using the
elements whose product we are forming (see (\ref{eq:Wh Relad})
below).

For our universal example, we define a particular product in the
mapping cone of the generalized adjoint $\ad_\lambda :L(a, b) \to
\Der(L,  L(a,b); \lambda)$ defined above (\ref{eq:L(a,b)}). So
assume now that $(L,d_{L}) = \L(V; d_L)$ is a free DG Lie algebra.
Let $a$ and $b$ be  of degrees $p-1$ and $q-1$, respectively.  Then
recall $L(a, b) = \L(V, a,b, V^a, V^{b}; \partial_{a, b})$ and   the
suspensions $S_{a} \colon V \to V^{a}$ and $S_{b} \colon V \to
V^{b}$ extend to elements of $\left(\Der(L, L(a,b); \lambda),
D_{\lambda}\right)$ of degree $p$ and $q$, respectively, satisfying
$D_{\lambda}(S_{a}) = (-1)^{p-1}\ad_{\lambda}(a)$ and
$D_{\lambda}(S_{b}) = (-1)^{q-1}\ad_{\lambda}(b).$ So $((-1)^pa,
S_a)$ and $((-1)^qb, S_b)$ are cycles in degrees $p$ and $q$ of
$(\Rel(\ad_\lambda), \delta_{\ad_\lambda}).$ Define elements
$\Theta_{a}, \Theta_{b}$ of degrees $p$ and $q$  in the DG Lie
algebra $(\Der\big(L(a, b)\big), D)$ by setting
\begin{equation} \label{eq:Theta}
 \Theta_{x}(v) = S_{x}(v) \hbox{\ \ and \ \ }
 \Theta_{x}(a) = \Theta_{x}(b) = \Theta_x(V^a) = \Theta_x(V^b) = 0 \end{equation}
 for $x = a,b$ and $v \in V.$
Note that $\Theta_{x} \circ \lambda = S_{x} \in \Der(L, L(a,b); \lambda).$
From the previous step, we set
$$ \left\{ \Theta_a, \Theta_b \right\} = (-1)^{p+1} \left[ \Theta_a, D(\Theta_b) \right] \in
\Der_{p+q-1}(L(a, b))$$ and observe that $\left\{ \Theta_{a},
\Theta_{b} \right\} \circ \lambda \in \Der_{p+q-1}(L, L(a,b);
\lambda).$ Now define
\begin{equation} \label{eq:universal example} \lbr \left((-1)^pa, S_a \right), \left( (-1)^qb, S_b \right) \rbr =_{\mathrm{def}}
\left( (-1)^q [a, b], \left\{ \Theta_a, \Theta_b \right\} \circ
\lambda \right). \end{equation} Observe that the right-hand side is
a cycle of $(\Rel(\ad_\lambda), \delta_{\ad_\lambda})$ of degree
$p+q-1$. This is the universal example of a Whitehead product
mentioned above.

\begin{remark} \label{rem:S[a,b]}
Taking $(L,d_{L}) = (\mathcal{L}_{X}, d_{X})$ to be the Quillen model of a simply connected
complex $X$, we see that
$$ \left\{ \Theta_{a}, \Theta_{b} \right\} \circ \lambda \in
\Der_{p+q-1}(\mathcal{L}_{X}, \mathcal{L}_{X}(a,b); \lambda)$$
satisfies the boundary condition $(\ref{eq:boundary})$.  Setting
$S_{[a,b]} =  \left\{ \Theta_{a}, \Theta_{b} \right\} \circ \lambda$
completes the description of the Quillen minimal model of $\eta
\times 1 \colon S^{p+q-1} \times X \to (S^{p} \vee S^{q}) \times X.$
\end{remark}

Finally, we turn to the mapping cone $(\Rel(\ad_\psi),
\delta_{\ad_\psi})$  of the generalized adjoint corresponding to a
DG Lie algebra map $\psi \colon (L, d_L) \to (K, d_K)$ with $L =
\L(V)$ free. Suppose given two $\delta_{\ad_{\psi}}$-cycles,
 $$ \zeta_{a} = (\chi_{a}, \theta_{a}) \in \Rel_{p}(\ad_{\psi})
\hbox{\, and \, }
\zeta_{b}=(\chi_{b}, \theta_{b})  \in \Rel_{q}(\ad_{\psi}).$$   The
pair $\zeta_{a},\zeta_{b}$
induce  a
DG Lie algebra map
\begin{equation} \label{eq:induced}\induced \colon (L(a,b), d_{L(a,b)})
\to (K, d_{K})\end{equation}
defined, on the basis of $L(a,b)$,  as:
$$  \induced(x) = (-1)^{|x|+1}\chi_{x}, \ \
\induced(v) = \psi(v) \hbox{ \, and \, }
\induced(S_{x}(v))
= \theta_{x}(v)
$$
for $x = a,b$ and $v \in V.$ Note that this map commutes with differentials
on generators of the form $S_x(v)$ since $ \induced \circ S_x$ and $\theta_x$
agree on $L$ as $\psi$-derivations.
Define \begin{equation}\label{eq:zpair} \{ \zeta_{a}, \zeta_{b} \} =_{\mathrm{def}}
\induced \circ \{ \Theta_{a},
\Theta_{b} \} \circ \lambda \in \Der_{p+q-1}(L, K; \psi).
\end{equation}
We have the following result concerning the iteration of this pairing.
\begin{proposition} \label{pro:3} Let $ \zeta_{a}, \zeta_{b}, \zeta_{c} \in \Rel(\ad_{\psi})$
be $\delta_{\ad_\lambda}$-cycles of degree $p, q$ and  $r$, respectively.  Then
$$ \left\{ \{ \zeta_{a}, \zeta_{b} \}, \zeta_{c} \right\} =
\inducedddd \circ \left\{ \{ \Theta_{a},
\Theta_{b} \}, \Theta_c \right\} \circ \lambda \in \Der_{p+q+r-2}(L, K; \psi).$$
Here $\inducedddd \colon (L(a,b,c), d_{L(a,b,c)})
\to (K, d_{K})$ is defined as  the obvious  extension of the definition of $\induced$ given by (\ref{eq:induced}).
\end{proposition}
\begin{proof} Let $\zeta_z = \{ \zeta_a, \zeta_b \} \in \Rel_{p+q-1}(L, K; \psi)$.
Let $\gamma_z = ((-1)^{q} [a,b],\{ \Theta_a, \Theta_b \} \circ
\lambda) \in \Rel_{p+q-1}(L, L(a,b,c); \lambda)$ be the
$\delta_{\ad_\lambda}$-cycle as in (\ref{eq:universal example}).
Define a DG Lie algebra map
$$\phi_z \colon L(z, c) \to L(a, b, c)$$
by setting $\phi_z(v) = v$, $\phi_z(c) = c$, $\phi_z(S_c(v)) =
S_c(v)$, $\phi_z(z) = (-1)^{q-1} [a,b]$, and $\phi_z(S_z(v)) = \{
\Theta_a, \Theta_b \}\circ\lambda(v)$.  This is readily checked to
define a DG map: use the fact that $\gamma_z$ is a cycle to check on
generators $S_z(v)$. Then we have a commutative diagram
$$ \xymatrix{ L(z, c) \ar[rr]^{\phi_z} \ar[rd]_{(\zeta_{z} \mid \zeta_{c})_{\!  \psi}} && L(a, b, c) \ar[ld]^{\inducedddd} \\
& K & } .$$ The needed identity now follows directly from the
observation that
$$\phi_z \circ \{\Theta_z, \Theta_c \}\circ  \lambda =
\left\{ \{ \Theta_{a}, \Theta_{b} \}, \Theta_c \right\} \circ
\lambda \in \Der_{p+q+r-2}(L, L(a,b,c); \lambda).\qed$$
\renewcommand{\qed}{}\end{proof}

We  obtain a bilinear pairing $\lbr \, , \, \rbr$ on pairs of  cycles  of
$\Rel(\ad_{\psi})$:
\begin{equation} \label{eq:Wh Relad} \lbr \zeta_{a},\zeta_{b} \rbr  = \lbr (\chi_{a}, \theta_{a}), (\chi_{b}, \theta_{b}) \rbr
=_{\mathrm{def}}
\left((-1)^{p}[\chi_{a}, \chi_{b}], \{\zeta_{a}, \zeta_{b} \}
\right).\end{equation}

\begin{remark}\label{rem:zeta}
The diagram of chain maps
$$\xymatrix{   L(a,b) \ar[rr]^{ \! \! \! \! \! \! \! \! \ad_\lambda} \ar[d]_{\induced} && \Der(L, L(a,b); \lambda)
\ar[d]^{\inducedd} \\
K \ar[rr]^{\!  \! \! \! \! \! \! \ad_{\psi}} &&  \Der(L, K; \psi) }
$$
commutes, inducing a chain map of mapping cones
$$ \zeta \colon (\Rel(\ad_\lambda), \delta_{\ad_\lambda}) \to (\Rel(\ad_\psi), \delta_{\ad_\psi}).$$
The  map   $\zeta$  carries $\lbr \left((-1)^pa, S_a \right), \left( (-1)^qb, S_b \right) \rbr,$
 the universal example
of the Whitehead product  defined by (\ref{eq:universal example}),  to the  cycle
$ \lbr   \zeta_a  ,  \zeta_b  \rbr$  defined by (\ref{eq:Wh Relad}).
\end{remark}

Our main result in this section is:
\begin{theorem}  \label{thm:Wh main}
    Let $\psi \colon (L, d_{L})
    \to (K, d _{K})$ be a DG Lie algebra map with   $L$ free and with
    adjoint  map $\ad_{\psi} \colon
    (K, d_{K}) \to (\Der(L, K; \psi), D_{\psi})$.
The bilinear pairing $\lbr \, , \, \rbr$ defined on cycles of
$(\Rel(\ad_{\psi}), \delta_{\ad_{\psi}})$ by {\em (\ref{eq:Wh
Relad})} induces a Whitehead product $[ \, , \, ]_{w}$ on
$H_{*}(\Rel(\ad_{\psi})).$
\end{theorem}
\begin{proof}
The fact that $\lbr \, , \, \rbr$  induces a pairing on cycles of
$(\Rel(\ad_{\psi}), \delta_{\ad_{\psi}})$ follows from
\remref{rem:zeta} and the observation immediately following
(\ref{eq:universal example}). Now we check that the Whitehead
identities (i)--(iii) are satisfied up to boundaries.
First, (i) is evident.   For (ii), we return
to (\ref{eq:universal example}) and write
$$\begin{aligned}%
\lbr \left((-1)^pa, S_a \right), &\left( (-1)^qb, S_b \right) \rbr =
\left( (-1)^q [a, b], \left\{ \Theta_a, \Theta_b \right\} \circ
\lambda \right)\\
&=\left( (-1)^{q+1+(p-1)(q-1)} [b, a], \left( (-1)^{pq} \left\{
\Theta_b, \Theta_a \right\} - D\left[ \Theta_a, \Theta_b \right]
\right) \circ \lambda \right)\\
&= (-1)^{pq} \left( (-1)^{p} [b, a], \left\{ \Theta_b, \Theta_a
\right\} \right) - \left(0, \big(D\left[ \Theta_a, \Theta_b \right]
\big) \circ \lambda \right).
\end{aligned}$$
In this last term, $D$ denotes the differential in
$\Der\big(L(a,b)\big)$; the identity is obtained from the first part
of the proof of \propref{prop:Jacobi}.  Now observe that, in
$\big(\Der(L, L(a,b); \lambda), D_\lambda\big)$, we have
$$\big(D\left[ \Theta_a, \Theta_b \right]
\big) \circ \lambda = D_\lambda\big( \left[ \Theta_a, \Theta_b
\right]  \circ \lambda \big),$$
since each expression agrees on every generator $v$ of $L$.
Consequently, in $\Rel(\ad_\lambda)$, we have
$$\delta_{\ad_\lambda} \big( 0, \left[ \Theta_a, \Theta_b \right]
\circ \lambda \big) = \big( 0, (D\left[ \Theta_a, \Theta_b \right])
\circ \lambda \big).$$
Returning to the above, we may write
$$\begin{aligned}
\lbr \left((-1)^pa, S_a \right), \left( (-1)^qb, S_b \right) \rbr =
(-1)^{pq} &\lbr \left( (-1)^qb, S_b \right), \left((-1)^pa, S_a
\right) \rbr \\ &- \delta_{\ad_\lambda} \big( 0, \left[ \Theta_a,
\Theta_b \right] \circ \lambda \big). \end{aligned}$$
From the observation of \remref{rem:zeta}, it now follows that the
pairing of (\ref{eq:Wh Relad}) of cycles of $\Rel(\ad_\psi)$
satisfies the Whitehead identity (ii) up to boundaries, and in
particular induces a pairing $H_*\big(\Rel(\ad_\psi)\big)$ that
satisfies (ii).

The proof of (iii) follows the same line of argument making use of
\propref{pro:3}.  Indeed, using it, we may amplify the diagram of
\remref{rem:zeta} into the following commutative diagram of chain
maps:
$$\xymatrix{ L(z,c) \ar[rrrr]^-{\ad_\lambda} \ar[dd]_{\phi_z}
\ar[rd]^{(\zeta_{z} \mid \zeta_{c})_{\!  \psi}}& & & & \Der(L,
L(z,c); \lambda)
\ar[ld]^{\big((\zeta_{z} \mid \zeta_{c})_{\!\psi}\big)_{\!{*}}} \ar[dd]^{(\phi_z)_*}\\
 & K \ar[rr]^-{\ad_{\psi}} &&  \Der(L, K; \psi)\\
 L(a,b,c) \ar[rrrr]_-{\ad_\lambda} \ar[ru]_{\inducedddd} & & & &\Der(L, L(a,b,c); \lambda)
\ar[ul]_{\big(\inducedddd\big)_{\!{*}}} \\
}
$$
which in turn induces chain maps of mapping cones of the horizontal
maps.  Each term in (iii) may thus be identified as the image of its
counterpart in the induced map of mapping cones given by the lower
trapezoid of this diagram.  Again the corresponding identity holds
up to a boundary in this lower $\Rel(\ad_\lambda)$, in which
$\lambda \colon L \to L(a,b,c)$. As in the preceding case, this
passes to the needed identity in $\Rel(\ad_\psi)$ by the chain maps
induced by the chain map $\inducedddd \colon L(a,b,c) \to K.$

It remains to show that, if either $\zeta_a$ or $\zeta_b$ is a
boundary, then  $\lbr \zeta_a, \zeta_b \rbr$ is a boundary also so
that the pairing passes to homology.  Suppose then that $\zeta_a =
\delta_{\ad_\psi}(\zeta_{c})$ is a boundary in
  $\Rel_{p}(\ad_\psi)$
so that  $\zeta_c = (\chi_{c}, \theta_c) \in \Rel_{p+1}(\ad_\psi)$ satisfies
$$d_K(\chi_c) = -\chi_{a} \hbox{ \ \  and  \ \ } D_\psi(\theta_c) = \theta_a -
\ad_\psi(\chi_c).$$
We show that $\lbr \zeta_{a}, \zeta_{b} \rbr$ bounds also.
To do this,  we would like to form the product $\lbr\zeta_{c}, \zeta_{b}
\rbr  \in \Rel_{p+q}(\ad_\psi)$ but our construction above requires
$\zeta_c$ to be a $\delta_{\ad_{\psi}}$-cycle.
To accomodate non-cycles, we  modify the  construction
of $L(a,b)$ as follows.  Define a DG Lie algebra
$ (\widehat{L}(a,b,c), \widehat{d}) =  \L(V, a,b, c, V^a, V^b, V^c;
\widehat{d})$ with $|c| = p,$ and with differential given by
$\widehat{d}(v) = d_L(v), \ \ \widehat{d}(a) =  \widehat{d}(b) =  0,
\ \ \widehat{d}(c) = -a $ and
$$ \widehat{d}(S_a(v)) = (-1)^{p-1}[a, v] + (-1)^pS_a(d_{L}(v))$$
$$ \widehat{d}(S_b(v)) = (-1)^{q-1}[b, v] + (-1)^qS_b(d_{L}(v))$$
$$\widehat{d}(S_c(v)) = (-1)^{p-1}[c, v] + S_a(v) + (-1)^{p+1}S_c(d_{L}(v))$$
for $v \in V.$
The formula for the boundary of $S_c$  gives  the relation
$$D_\lambda(S_c) =
(-1)^{p-1}\ad_\lambda(c) + S_a \in \Der(L, \hat{L}(a,b,c);
\lambda) $$
where $\lambda \colon L \to \hat{L}(a,b,c)$ is the inclusion.
Define $\Theta_{b}, \Theta_c$ in  $\Der(\hat{L}(a,b,c))$ as above:
$$\Theta_{x}(v) = S_{x}(v) \hbox{\, and \, } \Theta_{x}(y)  = \Theta_x(V^y)  = 0$$
for $x = b,c,$  $y = a, b, c$ and $v \in V.$
The classes $\zeta_a,  \zeta_b$ and $\zeta_c$  induce, as above,   a DG
Lie algebra map
$$ \varphi \colon  (\widehat{L}(a,b,c), d) \to (K,d_K)
\hbox{ \ \ with \ \ }$$
$$ \varphi(x) = (-1)^{|x|+1} \chi_x,
\varphi(c) = (-1)^{p} \chi_c, \,  \varphi(v) = \psi(v)  \hbox{\,
\and \, }  \varphi(S_y(v)) = \theta_y(v)$$ for $x = a,b$, $y = a,b,
c$ and $v \in V.$
Writing $\widehat{\lambda} \colon L \to \widehat{L}(a,b,c)$ for the inclusion,
a straightforward computation in $(\Rel\left(\ad_{\widehat{\lambda}}\right), \delta_{\ad_{\widehat{\lambda}}})$ shows that
$$ \delta_{\ad_{\widehat{\lambda}}} \left((-1)^q[c, b], - \{ \Theta_c, \Theta_b \} \circ \widehat{\lambda} \right)
= \left((-1)^q[a, b],  \{ \Theta_a, \Theta_b \} \circ \widehat{\lambda} \right).$$
That is, the universal example of a Whitehead product, constructed
now in the complex $(\Rel\left(\ad_{\widehat{\lambda}}\right),
\delta_{\ad_{\widehat{\lambda}}})$, is a boundary there. As in
Remark \ref{rem:zeta}, the map $\varphi$ induces a chain map $\phi
\colon (\Rel\left(\ad_{\widehat{\lambda}}\right),
\delta_{\ad_{\widehat{\lambda}}}) \to (\Rel(\ad_\psi),
\delta_{\ad_\psi})$.   As $\phi \left((-1)^q[a,b], \{ \Theta_a,
\Theta_b \} \circ \widehat{\lambda} \right) = \lbr \zeta_a, \zeta_b \rbr,$ it
follows that the latter is a boundary.
\end{proof}

Finally, we obtain a Whitehead product on $H_*(\Der(L, K;
\psi))$ when $L$ is free as a direct consequence of the above.
Suppose $\theta_a \in \Der_p(L, K; \psi)$ and $\theta_b \in \Der_q(L, K; \psi)$
are $D_\psi$-cycles.    Set $\zeta^*_a = (0, \theta_a) \in
\Rel_{p}(\ad_\psi)$ and $\zeta^*_b =
(0, \theta_b) \in \Rel_q(\ad_\psi)$. Both are $\delta_{\ad_\psi}$-cycles.
Thus we can write
$$ \lbr \zeta^*_{a}, \zeta^*_{b} \rbr = \left( 0,
\{\zeta^*_a, \zeta^{*}_{b} \}\right)   \in \Rel_{p+q-1}(\ad_\psi).$$
We define
\begin{equation} \label{eq:Wh Relad*} \lbr \theta_{a}, \theta_{b} \rbr =_{\mathrm{def}} \{ \zeta^*_a ,
\zeta^*_b \} \  \in \Der_{p+q-1}(L, K; \psi).\end{equation}
We then obtain:

\begin{corollary} \label{cor:Wh based}
    Let $\psi \colon (L, d_{L})
    \to (K, d _{K})$ be a DG Lie algebra map with $L$ free.
    The bilinear pairing $\lbr \, , \, \rbr$ defined on cycles
    of $(\Der(L, K; \psi), \delta_{\psi})$ by {\em (\ref{eq:Wh Relad*})} induces a
    Whitehead product $[ \, , \, ]_{w}$ on $H_{*}(\Der(L, K;
    \psi)).$    \qed
\end{corollary}

\section{Iterated Products of the Universal Example} \label{sec:iterated}
In this section, we continue with our algebraic development and
record some formulas and results that will be useful for our
applications. Let $(L, d)$ be a given free DG Lie algebra and
$(L(a,b), \partial_{a,b})$ the associated DG Lie algebra for $|a| =
p-1, |b| = q-1.$ We look in detail at the universal example of a
Whitehead product
 $$\{ \Theta_{a}, \Theta_{b} \} \circ \lambda \in \Der_{p+q-1}(L,
 L(a,b); \lambda).$$
 Begin in the DG Lie algebra $\big( \Der\big(L(a,b)\big), D \big).$
We are interested in the restriction of derivations $\Theta \in
\Der(L(a,b))$ to $L \subseteq L(a,b),$ i.e., $\Theta \circ \lambda
\in \Der(L, L(a,b); \lambda)$ where $\lambda \colon L \to L(a,b)$ is
the inclusion.

Recall from (\ref{eq:Theta}) the definitions of $\Theta_a, \Theta_b$
of degree $p$ and $q$ in $\big(\Der\big(L(a,b)\big), D\big).$ From
the definitions, we have that
$$
D_\lambda(\Theta_a\circ\lambda) = (-1)^{p-1} \ad_{\lambda}(a) \qquad
\textrm{and} \qquad D_\lambda(\Theta_b\circ\lambda) =
(-1)^{q-1}\ad_{\lambda}(b).
$$
Amongst the terms that occur in $\{ \Theta_{a}, \Theta_{b} \} \circ
\lambda(v)$, we note that $\partial_{a,b}\circ \Theta_b \circ
\Theta_a \circ\lambda(v)= 0$, whereas, e.g.~$\Theta_b\circ
\partial_{a,b}\circ \Theta_a \circ \lambda(v)$ is generally
non-zero. Using these facts, we obtain that

$$
\begin{aligned} \{ \Theta_a, \Theta_b \}\circ\lambda(v) & = (-1)^{p+1}  \left[ \Theta_a,
  D(\Theta_b) \right]\circ\lambda(v)\\
& = (-1)^{p+q} \Theta_a \circ \ad_{\lambda}(b)(v) + (-1)^{pq} D(\Theta_b) \circ \Theta_a \circ\lambda(v)  \\
& = (-1)^{p+q} \Theta_a \circ \ad_{\lambda}(b) (v) + (-1)^{pq+q+1} \Theta_b \circ \partial_{a,b} \circ \Theta_a \circ\lambda(v) \\
& =  (-1)^{p+q} \Theta_ a \circ \ad_{\lambda}(b)(v) + (-1)^{pq+p+q}
\Theta_b \circ \ad_{\lambda}(a)(v) \\
& \    + (-1)^{(p+1)(q+1)} \Theta_b \circ \Theta_a \circ\lambda (dv),
\end{aligned}
$$
yielding finally
\begin{equation}\label{eq:Wh prod on v}
\begin{aligned} \{ \Theta_a, \Theta_b \}\circ\lambda(v) & =   (-1)^{q(p+1)} [b, S_a(v)] + (-1)^{p}
[a, S_b(v)] \\
& \    + (-1)^{ (p+1)(q+1)} \Theta_b \circ \Theta_a \circ\lambda (dv).
\end{aligned}
\end{equation}

The formulae of the previous section may be extended to iterated
Whitehead products.  We sketch this here.  Suppose given
$\delta_{\ad_{\psi}}$-cycles $\zeta_1, \ldots, \zeta_n \in
\Rel_{*}(\ad_{\psi})$, with each $\zeta_{i} = (\chi_{a_i},
\theta_{a_i}) \in \Rel_{p_i}(\ad_{\psi})$, with $n \geq 2$ and each
$p_i \geq 2$. Let $a_1, \ldots, a_n$ be of degrees $p_1-1, \ldots,
p_n-1$, and define elements $\Theta_{a_i}$ of degree $p_i$ in the DG
Lie algebra $\big(\Der\big(L(a_1, \ldots, a_n)\big), D\big)$ as at
(\ref{eq:Theta}) above. Thus, we have derivations $\Theta_{a_i}
\circ \lambda \in \Der(L, L(a_1, \ldots, a_n); \lambda)$ that
satisfy $D_\lambda(\Theta_{a_i} \circ \lambda) = (-1)^{p_i-1}
\ad_\lambda(a_i)$. Write $$w(a_1, \ldots, a_n) = \left[ \left[
\ldots \left[[a_1, a_2], a_3\right] \ldots a_{n-1}\right], a_n
\right]$$ for the ``left-justified" iterated bracket in $L(a_1,
\ldots, a_n)$ and, similarly,
$$w(\Theta_{a_{1}}, \ldots,  \Theta_{a_{n}}) = \left\{ \left\{ \ldots
\left\{\{\Theta_{a_{1}}, \Theta_{a_{2}}\}, \Theta_{a_{3}}\right\}
\ldots \Theta_{a_{n-1}}\right\}, \Theta_{a_{n}} \right\}$$ for the
iterated product of derivations, using the pairing of (\ref{eq:L
product}) in $\Der(L(a_1, \ldots, a_n))$. Then
\begin{equation}\label{eq:universal example iterated}
\left(\pm w(a_1, \ldots, a_n), \pm w(\Theta_{a_{1}}, \ldots,
\Theta_{a_{n}})\circ\lambda \right)
\end{equation}
is a cycle of $(\Rel(\ad_\lambda), \delta_{\ad_\lambda})$ that is
the universal example of iterated Whitehead products of this form.

The $\zeta_{i}$ induce a DG Lie algebra map as in (\ref{eq:induced})
$(\zeta_I)_\psi \colon L(a_1, \ldots, a_n) \to (K, d_{K}).$  We
define the iterated Whitehead product as
\begin{equation}\label{eq:Whitehead product iterated}
\begin{aligned}
\lbr  \lbr  \ldots \lbr \lbr \zeta_1, \zeta_2\rbr &, \zeta_3\rbr
\ldots \zeta_{n-1}\rbr , \zeta_n \rbr \\
 & = \left(\pm
(\zeta_I)_\psi\big(w(a_1, \ldots, a_n)\big), \pm ((\zeta_I)_\psi)_*
w(\Theta_{a_{1}}, \ldots, \Theta_{a_{n}})\circ\lambda \right).
\end{aligned}
\end{equation}

We  conclude this section by observing that the Whitehead products we have
constructed for a DG Lie algebra map are invariant under
quasi-isomorphisms in the second variable.

\begin{theorem} \label{thm:lifting}
 Let $\psi \colon (L,d_L) \to (K,d_K)$ be a map between connected DG Lie
algebras with $(L,d_L)$ finitely generated and minimal.  Suppose $\phi \colon (K, d_K) \to
(K', d_{K'})$ is a surjective DG Lie algebra map such that $H(\phi)
\colon H_*(K) \to H_*(K')$ is an isomorphism.  Then composition with $\phi$ induces
isomorphisms $$ H_*(\Der(L, K; \psi)) \cong H_*(\Der(L, K'; \phi
\circ \psi)) \hbox{ \ \ and \ \ } H_*(\Rel(\ad_\psi)) \cong
H_*(\Rel(\ad_{\phi \circ \psi})).$$ Further,
these are isomorphisms of
Whitehead algebras with all spaces  equipped with the Whitehead products   constructed
in \thmref{thm:Wh main} and \corref{cor:Wh based}.
\end{theorem}

\begin{proof}
Write $\psi' = \phi\circ\psi$; for a cycle $\zeta_a = (\chi_a,
\theta_a) \in \Rel(\ad_\psi)$, use $\zeta'_a$ to denote the
corresponding cycle $(\phi(\chi_a), \phi_*\circ\theta_a) \in
\Rel(\ad_{\psi'})$.  Then $\phi\circ (\zeta_{a} \mid \zeta_{b})_{\!
\psi} = (\zeta'_{a} \mid \zeta'_{b})_{\!  \psi'}\colon L(a,b) \to
K'$, and we have a commutative diagram as follows:
$$\xymatrix{ K \ar[rrrr]^-{\ad_\psi} \ar[dd]_{\phi}
& & & & \Der(L, K; \psi) \ar[dd]^{\phi_*}\\
 & L(a,b) \ar[lu]_{(\zeta_{a} \mid \zeta_{b})_{\!  \psi}}
 \ar[ld]^{(\zeta'_{a} \mid \zeta'_{b})_{\!  \psi'}}\ar[rr]^-{\ad_{\lambda}} &&
 \Der(L, L(a,b); \lambda)
\ar[ru]^{\big((\zeta_{a} \mid \zeta_{b})_{\!\psi}\big)_{\!{*}}}
\ar[rd]^{\big((\zeta'_{a} \mid \zeta'_{b})_{\!\psi'}\big)_{\!{*}}}\\
K' \ar[rrrr]_-{\ad_{\psi'}}  & & & &\Der(L, K';
\psi') \\
}
$$
From this, it is clear that the map $\phi$ induces a commutative
diagram of long exact homology sequences of the respective  adjoints
$\ad_\psi$ and $\ad_{\psi'}$, and also that the construction of our
Whitehead product is natural with respect to the maps induced by
composition with $\phi$.

Thus by the Five-Lemma it suffices to prove composition with $\phi$
induces an isomorphism
 $H_*(\Der(L, K; \psi)) \cong H_*(\Der(L, K'; \phi
\circ \psi)).$ The proof of this fact is an  adaptation of a
standard result for lifting maps with domain a minimal DG Lie
algebra (see \cite[Prop.22.11]{F-H-T}). We give the full details to
show composition by $\phi$ induces a surjection on homology.  The
proof of injectivity is similar and so we omit it.

Write $L = \L(V; d_L)$ where $V = \Q(x_1, \ldots, x_n)$ and the
$x_i$ are homogeneous of nondecreasing degree. Let $\theta' \in
\Der_p(L, K'; \phi \circ \psi)$ be a $D$-cycle where we write $D=
D_{\phi \circ \psi}$.   We define $\theta \in \Der_p(L, K; \psi)$
and a derivation $\theta'' \in \Der_{p+1}(L, K'; \phi \circ \psi)$
so that
\begin{equation}\label{eq:induction} D_\psi(\theta) = 0
\hbox{ \ \ and \ \ } \phi \circ \theta = \theta' + D(\theta'').
 \end{equation}
We define $\theta$ and $\theta''$ on our basis for $V$ by induction.

Observe that, since $d_L(x_1) = 0$,   $d_{K'}(\theta'(x_1)) =
D(\theta')(x_1) = 0$. Since $\phi \: K \to K'$ induces a homology
isomorphism  we can choose a $d_K$-cycle $\chi \in K$ such that
$\phi(\chi) = \theta'(x_1) + d_{K'}(\alpha)$ for some $\alpha \in
K'$. We set $\theta(x_1) = \chi$ and $\theta''(x_1) = \alpha.$

Now suppose $\theta(x_j)$ and $\theta''(x_j)$ are defined for $j <
r$, such that (\ref{eq:induction}) holds on the Lie subalgebra
$\L(x_1, \ldots, x_{r-1})$ of $L$. Set   $y =
(-1)^{p}\theta(d_L(x_r)) \in K.$ Applying  our induction hypothesis,
we see $d_K(y) = 0.$ Furthermore, since $D(\theta') = 0$, we have
  \begin{align*} d_{K'}(\theta'(x_r))
&=  (-1)^{p}\theta'(d_{L}(x_r)) & \\ & = \phi(y) +
(-1)^{p+1}D(\theta'')(d_L(x_r))
  & = \phi(y) +(-1)^{p+1}d_{K'}(\theta''(d_L(x_r))).
  \end{align*}
Thus $\phi(y)$ is a boundary in $(K', d_{K'})$ and so we can choose
$z \in K$ with $d_K(z) = y.$ Next note that $ \theta'(x_r) +(-1)^{p}
\theta''(d_L(x_r)) - \phi(z)$ is a $d_{K'}$-cycle. Thus, as above,
we can find a $d_K$-cycle $\overline{z}$ and $\alpha \in K'$ such
that
$$ \phi(\overline{z}) = \theta'(x_r) + (-1)^{p}\theta''(d_L(x_r)) - \phi(z) + d_{K'}(\alpha).$$
We put $\theta(x_r) = \overline{z} +z$ and $\theta''(x_r) = \alpha$
and (\ref{eq:induction}) is satisfied on $\L(x_1, \ldots, x_{r})$.
\end{proof}

\section{Whitehead Product Formulae for Function Space Components} \label{sec:formula}
In this section we return to the topological setting and prove our
main result, the identification of the Whitehead product in the
rational homotopy groups of a function space component.  Let $f
\colon X \to Y$ be a map between simply connected CW complexes of
finite type with $X$ now a finite complex. Let $\mathcal{L}_f \colon
(\mathcal{L}_X, d_X) \to (\mathcal{L}_Y, d_Y)$ be the Quillen
minimal model for $f.$ We first recall the identifications $$
\pi_p(\map(X, Y;f)) \otimes \Q \cong
H_p(\Rel(\ad_{\mathcal{L}_f}))$$ for $p > 1$ given  in
\cite[Th.3.1]{L-S2}.

The adjoint of a representative $a\colon S^p \to \map(X, Y; f)$ of a
homotopy class $\alpha \in \pi_p(\map(X, Y; f))$ is a map $A \colon
S^p \times X \to Y.$ By \thmref{thm:SxX}, the Quillen minimal model
for $A$ is a map
$$\mathcal{L}_A \colon \big(\mathcal{L}_X(a),
\partial_{a}\big) \to (\mathcal{L}_Y, d_Y).$$
Define $\theta_a \in \Der_p(\mathcal{L}_X, \mathcal{L}_Y;
\mathcal{L}_f)$   by setting $\theta_a(v) = \mathcal{L}_A(S_a(v))$
for $v \in V$ and extending as an $\mathcal{L}_f$-derivation. Then
$\chi_a  = (-1)^p\mathcal{L}_A(a)$ is a cycle of degree $p-1$ in
$\mathcal{L}_Y$, and $\zeta_a = (\chi_a, \theta_a) \in
\Rel_{p}(\ad_{\mathcal{L}_f})$ is a
$\delta_{\ad_{\mathcal{L}_f}}$-cycle. Set
\begin{equation} \label{eq:Phi'}\Phi'(\alpha) = \left\langle \zeta_{a} \right\rangle \in
H_{p}(\Rel(\ad_{\mathcal{L}_f})).\end{equation}
The map $\Phi'$ is then a homomorphism whose rationalization
$ \Phi \colon   \pi_p(\map(X, Y;f)) \otimes
\Q \to H_p(\Rel(\ad_{\mathcal{L}_f}))$
is an isomorphism for $p \geq 2$.

Given two homotopy classes $\alpha \in \pi_p(\map(X, Y; f))$ and
$\beta \in \pi_q(\map(X, Y; f))$, their Whitehead product $\gamma =
[\alpha, \beta]_w$ has adjoint $C$ given by
$$\xymatrix{S^{p+q-1}\times X \ar@/_2pc/[rr]_{C} \ar[r]^-{\eta \times 1} & (S^p \vee
S^q)\times X \ar[r]^-{(A\mid B)_f} & Y}$$ where $\eta$ is the
universal example of the Whitehead product. Let  $\zeta_a = (\chi_a,
\theta_a) \in \Rel_p(\ad_{\mathcal{L}_f})$ and $\zeta_b = (\chi_b, \theta_b)
\in \Rel_q(\ad_{\mathcal{L}_f})$ satisfy $\langle \zeta_a \rangle =
\Phi'(\alpha)$ and $\langle \zeta_b \rangle = \Phi'(\beta).$

\begin{lemma} \label{lemma:model1}
The map
$$ (\zeta_a \mid \zeta_b)_{\! \mathcal{L}_f} \colon
\big(\mathcal{L}_X(a,b), \partial_{ a,b} \big) \to
(\mathcal{L}_Y, d_Y)$$
defined by (\ref{eq:induced})
is the Quillen minimal model for $$(A\mid
B)_f \colon (S^p \vee
S^q)\times X \to Y.$$
\end{lemma}

\begin{proof}
Denote by $\QL$ the composite of the Sullivan and the Quillen
functors. That is, we write $\QL(Z)$ to denote the DG Lie algebra
obtained by applying the Quillen functor to the coalgebra dual of
$A^*(Z)$, which is the Sullivan functor applied to $Z$. (See \cite[Sec. 22(e)]{F-H-T} or \cite[I.1(7)]{Tan}.) For a space
$Z$, denote by $\eta_Z\colon \mathcal{L}_Z \to \QL(Z)$ the Quillen
minimal model of $Z$. To establish the Lemma, we want the diagram
$$\xymatrix{
\mathcal{L}_X(a,b) \ar[rr]^-{(\zeta_a \mid
\zeta_b)_{\!\mathcal{L}_f}}
 \ar[d]_{\eta_{(S^p \vee S^q)\times X}}
& &\mathcal{L}_Y \ar[d]^{\eta_{Y}}\\
\QL\big((S^p \vee S^q)\times X\big) \ar[rr]_-{\QL\big((A\mid
B)_f\big)} & &\QL(Y). }$$
to be homotopy commutative, in the DG lie algebra sense. Following
\cite[II.5.(20)]{Tan}, this means we seek a DG Lie algebra map
$\mathcal{H} \colon \mathcal{L}_X(a,b) \to (t, dt) \otimes \QL(Y)$
such that $p_0\circ \mathcal{H} = \eta_{Y} \circ (\zeta_a \mid
\zeta_b)_{\!\mathcal{L}_f}$ and $p_1\circ\mathcal{H} =
\QL\big((A\mid B)_f\big) \circ  \eta_{(S^p \vee S^q)\times X}$. Now
we have a pushout of DG Lie algebras
$$\xymatrix{\mathcal{L}_X \ar[r]^-{\lambda_a} \ar[d]_-{\lambda_b} & \mathcal{L}_X(a)
\ar[d]^{\overline{\lambda_b}}\\
\mathcal{L}_X(b) \ar[r]_-{\overline{\lambda_a}} &
\mathcal{L}_X(a,b),}$$
where the maps $\lambda_a,\lambda_b, \overline{\lambda_a}, \overline{\lambda_b}$ are the appropriate inclusions. 
Notice that our desired minimal model $(\zeta_a \mid \zeta_b)_{\!
\mathcal{L}_f}$ is exactly the pushout of the minimal models
$\mathcal{L}_A \colon \mathcal{L}_X(a) \to \mathcal{L}_Y$ and
$\mathcal{L}_B \colon \mathcal{L}_X(b) \to \mathcal{L}_Y$.  We will
obtain our homotopy $\mathcal{H}$ by pushing out homotopies from
$\mathcal{L}_X(a)$ and $\mathcal{L}_X(b)$.  To this end, suppose
that we have our chosen minimal model $\mathcal{L}_f \colon
\mathcal{L}_X \to \mathcal{L}_Y$, and a DG Lie algebra homotopy
$\mathcal{H}_f \colon \mathcal{L}_X \to (t, dt) \otimes \QL(Y)$ that
satisfies $p_0\circ \mathcal{H}_f = \eta_{Y} \circ \mathcal{L}_f$
and $p_1\circ\mathcal{H}_f = \QL(f)\circ \eta_{X}$.

As in the proof of \cite[Prop.A.3]{L-S2}, we may assume that the
following cube is (strictly) commutative:
\begin{displaymath}
\xymatrix{ & \mathcal{L}_{X}(a)
 \ar[rr]^-{\overline{\lambda_b}}
\ar'[d]^-{\eta_{S^p\times X}}[dd] & & \mathcal{L}_{X}(a,b)
\ar[dd]^-{\eta_{(S^p\vee S^q)\times X}} \\
\mathcal{L}_{X} \ar[ru]^{\lambda_a} \ar[rr]_(0.63){ \lambda_b}
 \ar[dd]_{\eta_X}& &
\mathcal{L}_{X}(b) \ar[ru]_{\overline{\lambda_a}}
\ar[dd]^(0.3){\eta_{S^q\times X}}\\
 &
\QL(S^p\times X) \ar'[r]_(0.7){\QL(i_1\times 1)}[rr] & &
\QL((S^p\vee S^q)\times X)  \\
\QL(X) \ar[ru]^{\QL(i_2)} \ar[rr]_{\QL(i_2)} & & \QL(S^q\times X)
\ar[ru]_{\QL(i_2\times 1)} }
\end{displaymath}
Next, in the diagram
$$\xymatrix{\mathcal{L}_X \ar[r]^-{\lambda_a} \ar[d]_{\eta_X} & \mathcal{L}_X(a)
\ar[d]_{\eta_{{S^p}\times X}} \ar[r]^-{\mathcal{L}_A} &
\mathcal{L}_Y
\ar[d]^{\eta_Y}\\
\QL(X) \ar[r]_-{\QL(i_2)} & \QL(S^p\times X) \ar[r]_-{\QL(A)} &
\QL(Y),}$$
the left-hand square commutes and furthermore, the right-hand square
commutes up to homotopy, but we may assume that the homotopy
$\mathcal{H}_a \colon \mathcal{L}_X(a) \to (t, dt) \otimes \QL(Y)$
extends the homotopy $\mathcal{H}_f$, that is, that we have
$\mathcal{H}_a \circ \lambda_a = \mathcal{H}_f$.  This last
assertion is easily justified by adapting the usual lifting lemma:
rather than lift $\QL(A)\circ \eta_{{S^p}\times X}$ through the
quasi-isomorphism $\eta_Y$ starting with the elements of lowest
degree in $\mathcal{L}_X(a)$, we may start with the lift already
defined on $\mathcal{L}_X$ as $\mathcal{L}_f$, with $\eta_Y\circ
\mathcal{L}_A = \eta_Y\circ \mathcal{L}_f$ and $\QL(A)\circ
\eta_{S^p\times X} = \QL(f)\circ \eta_X$ homotopic by
$\mathcal{H}_f$ when resticted to $\mathcal{L}_X$. (See
\cite[Prop.10.4]{G-M} for the corresponding result in the DG algebra
setting.) We argue similarly with $b$ and $B$ replacing $a$ and $A$
respectively. This gives us the pushout
\begin{displaymath}
\xymatrix{ \mathcal{L}_X  \ar[d]_{\lambda_b}
\ar[r]^(.44){\lambda_a}&
\mathcal{L}_X(a) \ar[d]_{\overline{\lambda_a}} \ar@/^/[ddr]^{\mathcal{H}_a}&\\
\mathcal{L}_X(b)\ar[r]^(.44){\overline{\lambda_b}} \ar@/_/[drr]_{\mathcal{H}_b}
& \mathcal{L}_X(a,b) \ar@{.>}[dr]^{\mathcal{H}}&\\
&&(t, dt) \otimes \QL(Y)\\}
\end{displaymath}
which defines $\mathcal{H}$. We check that $\mathcal{H}$ has the
desired properties.
\end{proof}

We next give an official statement of our work in
\secref{sec:Gamma}. Define   $$\Gamma \colon \big(\mathcal{L}_X(c),
\partial_{c}\big) \to \big(\mathcal{L}_X(a,b), \partial_{a,b}\big)$$
by setting $\Gamma (\chi) = \chi$ for $\chi \in \mathcal{L}_X,
\Gamma(c) = (-1)^{p-1}[a, b]$ and
$$\Gamma(S_c(v)) =  \left\{ \Theta_{a}, \Theta_{b} \right\} \circ
\lambda(v)$$
where   $\Theta_{a}, \Theta_{b} \in \Der(\mathcal{L}_{X}(a,b))$
are as defined in (\ref{eq:Theta}) and $\lambda \: \mathcal{L}_{X} \to
\mathcal{L}_{X}(a,b)$ is the inclusion.

\begin{lemma} \label{lemma:model2} The map $$\Gamma \colon
\big(\mathcal{L}_X(c), \partial_c\big) \to
\big(\mathcal{L}_X(a,b), \partial_{a,b)}\big)$$ is the Quillen model
for $$\eta\times 1 \colon  S^{p+q-1} \times X \to (S^{p} \vee S^{q}) \times X.$$
\end{lemma}
\begin{proof}
We have a commutative  diagram
$$\xymatrix{\left( \mathcal{L}_X(c), \partial_c\right) \ar[d]
\ar[r]^-{\Gamma} & \left(\mathcal{L}_X(a, b),
\partial_{a,b}\right)\ar[d]\\
\L(c; 0)\oplus (\mathcal{L}_X, d_X) \ar[r]_-{\phi} & \L(a,b;
0)\oplus (\mathcal{L}_X, d_X)}
$$
where the vertical maps are the projections and $\phi(c) =
(-1)^{p-1}[a, b]$ and $\phi(\chi) = \chi$ for $\chi \in
\mathcal{L}_X$. Since $\phi$ is evidently a (non-minimal) Quillen
model for $\eta \times 1$, the result follows from uniqueness of the
Quillen model of a map.
\end{proof}

Combining these facts we obtain our identification.

\begin{theorem} \label{thm:map}
 Let $f \colon X \to Y$ be a map between simply connected CW complexes
 of finite type with $X$ finite.
The map $$\Phi' \colon \pi_{p}(\map(X, Y; f)) \to
H_{p}(\Rel(\ad_{\mathcal{L}_{X}}))$$
defined for $p > 1$ by (\ref{eq:Phi'}) preserves Whitehead products
where the latter space has the Whitehead product given by
\thmref{thm:Wh main}.  Thus $\Phi'$ induces an isomorphism
$$ \pi_{*}(\map(X, Y;f)) \otimes \Q, \, \, [ \, , \, ]_{w} \cong
H_{*}(\Rel(\ad_{\mathcal{L}_{X}})), \, \, [ \, , \, ]_{w}$$
of rational Whitehead algebras in degrees $>1.$
\end{theorem}
\begin{proof}
With notation as above
and Lemmas \ref{lemma:model1} and \ref{lemma:model2}, $$
 (\zeta_a \mid \zeta_b)_{\mathcal{L}_f} \circ \Gamma\colon
(\mathcal{L}_X(c), \partial_{c}) \to  (\mathcal{L}_Y,
d_Y)$$
is the Quillen minimal model for the adjoint $C$ of $\gamma = [\alpha,
\beta] \in \pi_{p+q-1}(\map(X, Y;f)).$ Thus $\Phi'(\gamma)$ is
represented by the $\delta_{\ad_{\mathcal{L}_f}}$-cycle $\zeta_c =
(\chi_c, \theta_c) \in \Rel_{p+q-1}(\ad_{\mathcal{L}_f})$
with
$$\chi_c = (-1)^{p+q-1}   (\zeta_a \mid
\zeta_b)_{\mathcal{L}_f} \circ \Gamma (c) = (-1)^p[\chi_a, \chi_b] \in
\left(\mathcal{L}_Y\right)_{p+q-2}\hbox{ \ \
while \ \ }$$
$$ \theta_c(v)  =  (\zeta_a \mid
\zeta_b)_{\mathcal{L}_f} \circ \Gamma(S_c(v))  = (\zeta_a \mid
\zeta_b)_{\mathcal{L}_f} \left\{\Theta_{a}, \Theta_{b} \right\}
\circ \lambda(v) \in \Der_{p+q-1}(\mathcal{L}_X, \mathcal{L}_Y;
\mathcal{L}_f).$$ Thus $$\Phi'(\gamma) = \left\langle \lbr
\zeta_{a}, \zeta_{b} \rbr \right\rangle = \left[ \langle \zeta_{a}
\rangle, \langle \zeta_{b} \rangle \right]_{w}$$ by \thmref{thm:Wh
main} and the definition of the pairing $\lbr \, , \, \rbr$ at
(\ref{eq:Wh Relad}).
\end{proof}

We now apply the same line of reasoning to the case of the based
function space.  We first recall the homomorphism,
 \begin{equation} \label{eq:Psi'} \Psi' \colon \pi_p(\map_*(X, Y;f)) \to
\pi_p(\Der(\mathcal{L}_X, \mathcal{L}_Y; \mathcal{L}_f))
\end{equation}
from \cite[Th.3.1]{L-S2} inducing an isomorphism after rationalization for $p > 1$.
Given $\alpha \in \pi_p(\map_*(X, Y;f))$
 we have $\Psi'(\alpha) = \langle \theta_a \rangle$
where $ \theta_a  \in \Der_p(\mathcal{L}_X, \mathcal{L}_Y;
\mathcal{L}_f)$ is the $D_{\mathcal{L}_f}$-cycle given by $\theta_a
= \mathcal{L}_A \circ S_a$ where $\mathcal{L}_A \colon
(\mathcal{L}_X(a), \partial_{a}) \to (\mathcal{L}_Y, d_Y)$ is the
Quillen minimal model for the adjoint $A \colon S^p \times X \to Y$
of $\alpha.$  We prove
\begin{theorem} \label{thm:map*}
 Let $f \colon X \to Y$ be a map between simply connected CW complexes
 of finite type with $X$ finite.
The map $$\Psi' \colon \pi_{p}(\map_{*}(X, Y; f)) \to
H_{p}(\Der(\mathcal{L}_{X}, \mathcal{L}_{Y}; \mathcal{L}_{f}))$$
defined for $p > 1$ by (\ref{eq:Psi'}) preserves Whitehead products
where the latter space has the Whitehead product given by
\corref{cor:Wh based}.  Thus $\Psi'$ induces an isomorphism
$$ \pi_{*}(\map_{*}(X, Y;f)) \otimes \Q, \, \,  [ \, , \, ]_{w} \cong
H_{*}(\Der(\mathcal{L}_{X}, \mathcal{L}_{Y}; \mathcal{L}_{f})),\, \,  [ \, , \, ]_{w}$$
of rational Whitehead algebras in degrees $>1.$
\end{theorem}
\begin{proof}
Given  $\alpha,\beta \in \pi_*(\map_*(X, Y;f))$  of degrees $p$ and
$q$ with   Whitehead product $\gamma = [\alpha, \beta] \in
\pi_{p+q-1}(\map_*(X, Y;f)),$ the class  $\Psi'(\gamma)$  is
represented by
$$(\zeta^*_a \mid
\zeta^*_b)_{\mathcal{L}_f} \circ \Gamma  \circ S_c
 = \lbr \theta_{a}, \theta_{b} \rbr
    \in  \Der_{p+q-1}(\mathcal{L}_X, \mathcal{L}_Y;
    \mathcal{L}_f)$$
   by Lemmas \ref{lemma:model1} and \ref{lemma:model2}  and definition of the bilinear pairing $\lbr \, , \, \rbr$
  in (\ref{eq:Wh Relad*}).
    The result now follows from \corref{cor:Wh based}.
\end{proof}

\begin{remark}
For $\alpha_1, \ldots, \alpha_n \in \pi_{*}(\map(X, Y; f))$, let
$$w(\alpha_1, \ldots, \alpha_n) = \left[ \left[ \ldots \left[[\alpha_1, \alpha_2]_w,
\alpha_3\right]_w, \ldots, \alpha_{n-1}\right]_w, \alpha_n \right]_w$$
for their ``left-justified" iterated Whitehead product. The argument
above may easily be extended, using the algebraic universal
Whitehead product indicated in (\ref{eq:universal example
iterated}), and the
 topological universal example for such Whitehead products, namely
$$w(\iota_1, \ldots, \iota_n) \colon S^{p_1+\cdots + p_n -n+1} \to S^{p_1} \vee \cdots \vee
 S^{p_n }.$$
Using \propref{pro:3} and the above arguments,  we may show that the algebraic iterated Whitehead product indicated in  (\ref{eq:Whitehead product iterated})  in Section \ref{sec:iterated}
corresponds with
$w(\alpha_1, \ldots, \alpha_n)$ under the map $\Phi'$.  We have no
immediate need for this, and so we omit details.
\end{remark}

\section{Whitehead Length of Function Space Components} \label{sec:applications}
We apply our formulae to  study  the Whitehead length of function space components.
To begin, we make some remarks concerning the sensitivity of the invariants
$$\WL(\map(X, Y;f)) \hbox{\ \ and \, } \WL(\map_*(X, Y; f))$$
to the (homotopy class of the) map $f \colon X \to Y.$ For example,
in the case $X = Y$ and $f=1$, the space $\map(X, X; 1)$ is  a
topological monoid, and so $\WL(\map(X, X;1)) = 1.$  When $Y$ is an
$H$-space, so too is $\map(X, Y)$, and hence $\WL(\map(X, Y;f) = 1$
for any component.  Dually, if $X$ is a co-$H$-space, then
$\map_*(X, Y)$ is an $H$-space, and hence $\WL(\map_*(X, Y;f)) = 1$
for any component of the based mapping space. On the other hand, we
have the following   fact concerning the null-component which shows
that we may easily have an abundance of non-zero Whitehead products
in the free function space. Recall that we are only considering
$\pi_{\geq 2}\big(\map(X, Y;f)\big)$ here.

\begin{theorem} \label{thm:WL null component}
Let $Y$ be any space. Then
$$\max \{\WL(Y), \WL(\map_*(X, Y; 0)) \} \leq \WL(\map(X, Y; 0)).$$
If the universal cover of $Y$ has finite rational type then
$$ \WL_\Q(\map(X, Y;0))  = \WL_\Q(Y).$$
\end{theorem}

\begin{proof}
The first inequality follows from the evaluation fibration
$\map_*(X, Y; 0) \to \map(X, Y;0) \to Y$.  On the one hand, the
obvious section $s \colon Y \to \map(X, Y; 0)$ implies that $Y$ is a
retract of $\map(X, Y; 0)$. On the other hand, $s$ implies that the
fibre inclusion $\map_*(X, Y; 0) \to \map(X, Y; 0)$ induces an
injection on homotopy groups.

For the rational result, we start with a nice observation of
Brown-Szczarba \cite{B-S97b}: writing $\Omega_0 Y$ for the connected
component of the constant loop in $\Omega Y$, we have
$\Omega_0(\map(X, Y;0)) \approx \map(X, \Omega_0 Y; 0).$ Next, by
\cite[Th.4.10]{LPSS} $$\Hnil(\map(X, \Omega_0 Y; 0)) =
\Hnil(\Omega_0 Y)$$ where $\Hnil(G)$ of a loop-space $G$ denotes the
{\em homotopical nilpotency} of $G$ in the sense of Berstein-Ganea
\cite{B-G}.  Taking $Y = Y_\Q$ the result follows from the identity
$\Hnil(\Omega_0 Y_\Q ) = \WL_\Q(Y)$ \cite[Th.3]{Sal}.
\end{proof}

Recall that a simply connected space $Y$ is  {\em coformal}  if
there is a DG Lie algebra map $\rho \colon (\mathcal{L}_Y, d_Y) \to
(\pi_*(\Omega Y) \otimes \Q, 0)$ inducing an isomorphism on homology
(see \cite[p.334 Ex.7]{F-H-T}).

\begin{theorem} \label{thm:Y coformal}
Let $X$ be a finite simply connected CW complex and
$Y$ a simply connected coformal complex of finite type.  Then
for all $f \colon X \to Y$ we have
$$\max \{ \WL_\Q(\map_*(X, Y;f)), \WL_\Q(\map(X, Y;f)) \} \leq \WL_\Q(Y).$$
\end{theorem}

\begin{proof}
By \thmref{thm:lifting}, we may replace $(\mathcal{L}_Y, d_Y)$ by
$(H(\mathcal{L}_Y), 0) = (\pi_*(\Omega Y) \otimes \Q, 0)$ when we
apply \thmref{thm:map}.  We say that a cycle $\zeta = (\chi, \theta)
\in \Rel_*(\ad_{\mathcal{L}_f})$ is \emph{of length $\geq r$ in
$H(\mathcal{L}_Y)$} if $\chi \in H(\mathcal{L}_Y)$ is of bracket
length $\geq r$ and, when applied to a generator $v \in \L(V) =
\mathcal{L}_X$, $\theta(v)$ is also of bracket length $\geq r$ in
$H(\mathcal{L}_Y)$.  The result is proved by arguing that iterated
Whitehead products in $H_*(\Rel(\ad_{\mathcal{L}_f}))$ of length $r$
are represented by cycles of length $\geq r$ in $H(\mathcal{L}_Y)$.

To see this, consider two cycles $\zeta_a$ and $\zeta_b$, and
suppose that $\zeta_a$ is of length $\geq r$ in $H(\mathcal{L}_Y)$.
According to (\ref{eq:Wh prod on v}), $\{ \Theta_a, \Theta_b
\}\circ\lambda(v)$ is contained in the ideal of $\mathcal{L}_X(a,b)$
generated by $a$ and $S_a(v)$, and also is of length $\geq 2$.
Therefore, when the map $(\zeta_a\mid\zeta_b)_{\mathcal{L}_f}$ is
applied to it, we obtain an element of bracket length $\geq (r+1)$
in $H(\mathcal{L}_Y)$.  Likewise for the bracket $[a,b]$.  It
follows that a cocycle representative of $\lbr \zeta_a, \zeta_b
\rbr$ is of length $\geq (r+1)$ in $H(\mathcal{L}_Y)$.  An easy
induction using this completes the proof.
\end{proof}

We say a simply connected CW complex $X$ is a {\em rational
co-$H$-space} if $X_\Q$ is homotopy equivalent to a wedge of
spheres. We remarked above that, if $X$ is a co-$H$-space, then
Whitehead products vanish in any component of the based mapping
space.  The following result provides a large class of examples of
free function spaces with vanishing rational Whitehead products.

\begin{theorem} \label{thm:coformal abelian} Let $X$ be a finite rational
co-$H$-space and $Y$ a simply connected, coformal complex of finite
type. Suppose $f$ induces a surjection on rational homotopy groups.
Then $\WL_\Q(\map(X, Y;f)) = 1.$
\end{theorem}

\begin{proof}
Since $X$ is a rational co-$H$-space, the differential $d_X$ in the
Quillen minimal model for $X$ vanishes.  Since $Y$ is coformal, we
may replace $(\mathcal{L}_Y, d_Y)$ by $(H(\mathcal{L}_Y), 0)$, and
view $\mathcal{L}_f$ as a map $\mathcal{L}_X \to H(\mathcal{L}_Y)$,
when we apply \thmref{thm:map}.  Suppose given a pair $\zeta_a =
(\chi_a, \theta_a) \in \Rel_p(\ad_{\mathcal{L}_f})$ and $\zeta_b
=(\chi_b, \theta_b) \in \Rel_q(\ad_{\mathcal{L}_f})$ of
$\delta_{\ad_{\mathcal{L}_f}}$-cycles; with $\chi_a, \chi_b \in
H(\mathcal{L}_Y)$ and $\theta_a, \theta_b \in
\Der_{*}(\mathcal{L}_X, H(\mathcal{L}_Y); \mathcal{L}_f)$.  Using
the fact that
$$(\zeta_a\mid\zeta_b)_{\mathcal{L}_f} \colon \mathcal{L}_X(a,b) \to
H(\mathcal{L}_Y)$$
is a DG Lie algebra map, we see that $[\chi_a, \mathcal{L}_f(v)] =
\pm (\zeta_a\mid\zeta_b)_{\mathcal{L}_f}\big( \partial_{a,b} S_a (v)
\big) = 0$, for any $v \in \mathcal{L}_X$.  By assumption,
$\mathcal{L}_f$ is surjective and it follows that the bracket of
$\chi_a$ with any element of $H(\mathcal{L}_Y)$ is zero.  A similar
argument yields the same conclusion for $\chi_b$.  Finally, we
obtain from (\ref{eq:Wh prod on v}) that $\{ \Theta_a, \Theta_b
\}\circ\lambda(v) = \pm [b, S_a(v)] \pm [a, S_b(v)]$.  It follows
that in $H_*\big(\Rel(\ad_{\mathcal{L}_f})\big)$, we have
$$\lbr \zeta_a, \zeta_b \rbr = \big( \pm [\chi_a, \chi_b], \pm [\chi_a,
\theta_b(v)] \pm [\chi_b, \theta_a(v)]\big) = (0,0).$$
The result follows from \thmref{thm:map}.
\end{proof}

In \cite{Gan60}, Ganea proved $\WL(\map_*(X, Y; 0)) \leq \cat(X).$
We give a rational version of this inequality which applies to all
components. Recall the {\em  rational cone length} $\Cl_0(X)$ of a
space $X$ is   the least integer $n$ such that  $X$ has the rational
homotopy type of an $n$-cone (see \cite[p.359]{F-H-T}). Spaces of
rational cone length $1$ then correspond to rational co-$H$-spaces.

\begin{theorem} \label{thm:cone}
Let $X$ be a finite CW complex and
$Y$ a simply connected  complex of finite type. Then $\WL_\Q(\map_*(X,
Y;f)) \leq \Cl_0(X)$ for all maps $f \colon X \to Y.$
\end{theorem}

\begin{proof}
Let $n = \Cl_0(X).$ By \cite[Th.29.1]{F-H-T}, the underlying vector space $V$
of the Quillen minimal model
 of $X$ admits a filtration $\{ 0 \} \subset V(1) \subset V(2)
 \subset \cdots \subset V(n) = V$ where $d_X(V(i)) \subseteq
 \L(V(i-1)).$  The result is proved by arguing that iterated
Whitehead products in $H_*(\Der(\mathcal{L}_X, \mathcal{L}_Y;
\mathcal{L}_f))$ of length $r$ are represented by cycles that vanish
on $V(r-1)$.  We argue in a similar fashion to the proof of
\thmref{thm:Y coformal}.

Consider two cycles $\theta_a, \theta_b \in \Der(\mathcal{L}_X,
\mathcal{L}_Y; \mathcal{L}_f)$.  Recall that their Whitehead product
is represented by the image of the universal example
$$\{\Theta_{a}, \Theta_{b} \}\circ\lambda
\in \Der(\mathcal{L}_X, \mathcal{L}_X(a, b);\lambda)$$
under the map
$$\left((\zeta_{a}^*\mid\zeta_{b}^*)_{\mathcal{L}_f}
\right)_* \colon \Der(\mathcal{L}_X, \mathcal{L}_X(a, b); \lambda)
\to \Der(\mathcal{L}_X, \mathcal{L}_Y; \mathcal{L}_f).$$ Here
$\zeta_{a}^* = (0, \theta_{a})$ and $\zeta_{b}^* = (0, \theta_{b})$
in $\Rel(\ad_{\mathcal{L}_f})$.  In particular,
$(\zeta_{a}^*\mid\zeta_{b}^*)_{\mathcal{L}_f}$ maps $a$ and $b$ to
zero.

From (\ref{eq:Wh prod on v}), we see that $\{ \Theta_a, \Theta_b
\}\circ\lambda(v) \equiv \pm \Theta_b \circ \Theta_a (d_X(v))$
modulo terms in the ideal generated by $a$ and $b$.  Now assume that
$\Theta_a$ vanishes on $V(r)$.  Since $d_X(V(r+1)) \subseteq
 \L(V(r))$, we have that $\{ \Theta_a, \Theta_b
\}\circ\lambda$ vanishes on $V(r+1)$. An easy induction using this
completes the proof.
\end{proof}

We next give a complete calculation of the rational Whitehead length of function spaces in a special case.    Let   $X$ be  a simply
connected, finite complex and $S^n$ a sphere with $n \geq 2$.
When $n$ is odd, $S^n$ is a
rational $H$-space and hence so too is $\map(X, S^n)$.  It follows
that, after rationalization, each component of $\map(X, S^n)$ is
homotopy equivalent to the null component, which itself is an
$H$-space; in particular we have $\WL_\Q(\map(X, S^n; f)) =1$.
Identical remarks apply to the based function space.

When  $n$ is even, the rational homotopy types of components $\map(X, S^n; f)$ are more complicated. A complete description
   for    $X$ rationally
$(2n+1)$-co-connected is given by  M{\o}ller-Raussen \cite[Th.1]{M-R}.
We compute the rational Whitehead length of all components in both the based and free setting without   dimension restriction on $X$.

Since $S^n$ is a coformal space, $\WL_\Q(\map_*(X, S^n; f))$ and
$\WL_\Q(\map(X, S^n; f))$ are each equal to either $1$ or $2$---but
not \emph{a priori} equal to each other---by \thmref{thm:Y
coformal}. Suppose first that $H(f; \Q) = 0 \colon H^*(S^n; \Q) \to
H^*(X; \Q)$. Then the rationalization of $f$ factors through the
fibre $K(\Q, 2n-1)$ of the    Postnikov decomposition $K (\Q,2n-1)
\to (S^{n})_\Q \to K(\Q, n)$. This implies $f$ is a {\em rationally
cyclic map} (see \cite[Def.2.4 and Ex.4.4]{L-S1}). By
\cite[Th.3.7]{L-S3}, the evaluation fibration $\omega \colon \map(X,
S^n;f) \to S^n$ is then rationally fibre-homotopically trivial and,
from the long exact homotopy sequence,  we have an isomorphism of
Whitehead algebras:
$$\pi_*(\map(X, S^n; f)) \otimes \Q, [ \, , \, ]_w \cong
\left( \pi_*(\map_*(X, S^n; f)) \otimes \Q, [ \, , \, ]_w \right)
\oplus \left(\pi_*(S^n) \otimes \Q,  [ \, , \, ]_w \right).$$ Since
$\WL_\Q(S^n) = 2$, we have $\WL_\Q(\map(X, S^n; f)) = 2$ and
$\WL_\Q(\map_*(X, S^n; f))$ is equal to either $1$ or $2$ in this
case.

Suppose $H(f; \Q)  \neq 0.$
Write $(\mathcal{L}_{S^n}, d_{S^n}) = \L(u;
0)$ with $|u| = n-1$ and $\mathcal{L}_X = \L(V ;d_X).$
 The condition $H(f; \Q) \neq 0$ translated to Quillen models
implies there exists $v \in V_{n-1}$ with
  $\mathcal{L}_f(v) =u$.    This implies there are no
  cycles $(u, \theta) \in \Rel_n(\ad_{\mathcal{L} _f})$
because $\ad_{\mathcal{L}_f}(u)+D_{\mathcal{L}_f}(\theta)$ cannot
equal zero: note that $\theta(d_X(v)) = 0$ for degree reasons, and
then applied to $v$ we obtain that
$$\ad_{\mathcal{L}_f}(u)(v) + D_{\mathcal{L}_f}(\theta)(v) =
[u, \mathcal{L}_f(v)] + d_{S^n}\theta(v) \pm \theta(d_X(v)) = [u, u]
\neq 0 \in (\mathcal{L}_{S^n})_{2n-2}.$$  By the formula for the Whitehead
product in $H_*(\Rel(\ad_{\mathcal{L}_f}))$   (\thmref{thm:Wh main}),
we see
directly that the cycle $([u,u], 0) \in \Rel_{2n-1}(\ad_{\mathcal{L} _f})$
does not represent a Whitehead product.   Translating back, this means
$$\pi_*(\map(X, S^n; f)) \otimes \Q, [ \, , \, ]_w \cong \left( \pi_*(\map_*(X, S^n; f)) \otimes \Q, [ \, , \, ]_w \right)  \oplus \left(\Q([\iota, \iota]_w), 0 \right)$$
where $\iota  \in \pi_n(S^n)$ is nontrivial and $(\Q([\iota, \iota]_w), 0)$
denotes the abelian Whitehead algebra generated in degree $2n-1$.
Thus in this case $$ \WL_\Q(\map(X, S^n; f)) = \WL_\Q(\map_*(X, S^n; f)) = \hbox{ $1$ or $2.$}$$
In both cases,  the relevant question is the rational Whitehead length
of  the based function space. We address this question as an application of our formula:

\begin{theorem} \label{thm:sphere} Let $X$ be a finite, simply
connected CW complex and $f \colon X \to S^n$ a based map with $n$
even. Then
$ \WL_\Q(\map_*(X, S^n; f)) = 2 $ if and only if
there exists a pair $x, y \in H^{\leq n-2}(X; \Q)$ satisfying:
\begin{itemize} \item[(i)]  $xy \neq 0,$
\item[(ii)]    $xz \neq 0$ or $yz  \neq 0$ for   $z \in H^n(X; \Q) \implies H(f;\Q)(z) = 0  $
\hbox{and}
\item[(iii)]    $xy =wz$ for some $z \in H^n(X; \Q)$ and any $w    \implies H(f;\Q)(z) = 0  $
\end{itemize}
   Otherwise, $ \WL_\Q(\map_*(X, S^n; f)) =   1.$

   As for the free function space, if $H(f; \Q) =0$ then $\WL_\Q(\map(X, S^n;f))= 2$.
   Otherwise, $\WL_\Q(\map(X, S^n; f)) = \WL_\Q(\map_*(X, S^n; f)),$ as given above.
\end{theorem}
\begin{proof}
The results for the free function space follow from the discussion preceding the
statement of the theorem.  Thus we focus on the based function space and
so  the space $H_*(\Der(\mathcal{L}_X, \mathcal{L}_{S^n}; \mathcal{L}_f))$
 with Whitehead product given by \corref{cor:Wh based}.
 Write $\mathcal{L}_X = \L(V; d_X)$ and $\mathcal{L}_{S^n} = \L(u; 0).$
Given a homogeneous basis $\{v_1,  \ldots, v_s\}$ for $V = s^{-1}\widetilde{H}^*(X; \Q)$, we will assume the vectors are in nondecreasing order of degree.   If $H(f; \Q) \neq 0,$ we will  further assume that there is some basis vector $v_k \in V_{n-1}$  such that   $\mathcal{L}_f(v_k) = c_ku$
while $\mathcal{L}_f(v_i) =0 $ for any other basis element $v_i$ of degree $n-1.$ Here
$c_k \neq 0$.  For convenience, we allow the case  $c_k =0$ so that  $H(f; \Q) =0$  if and only if  $c_k =0.$

We recall the quadratic part of the differential $d_X$ is dual to the cup product in
$H^*(X;\Q)$  (see \cite[Sec.I.1.(10)]{Tan} and \cite[Sec.22e]{F-H-T}).
 Let $ \{x_1, \ldots, x_s\}$ be the corresponding additive basis of $\widetilde{H}^*(X; \Q)$, that is, $x_i = s(v_i).$
Given  any  $v \in V$
 we may write
\begin{equation} \label{eq:quad} d_X(v)  = \sum_{i \leq j}c_{ij}(v)[v_i, v_j] + \hbox{\, longer length terms}\end{equation}
with $c_{ij}(v) \in \Q$ and $c_{ii}(v) = 0$ for $v_i$ of even degree.
As a direct consequence of this duality we have that the cup product $x_i x_j = 0$ if and only if $c_{ij}(v) = 0$ for all $v \in V.$ We make use of this repeatedly below.

Let $\theta \in \Der_p(\mathcal{L}_X,\mathcal{L}_{S^n}; \mathcal{L}_f)$.
 Using (\ref{eq:quad}), we have
$$ D_{\mathcal{L}_f}(\theta)(v) =  \pm
\theta(d_X v) =  \pm  \sum
_{i\leq k} c_{ik}(v)[\theta(v_i), \mathcal{L}_f(v_k)]=  \pm   \sum_{i\leq k} c_k
 c_{ik}(v)[ \theta(v_i),u]. $$
It follows that $\theta$ is a cycle if $\theta(v_i) = 0$  for all $v_i$ in $V_{n-1-p}.$
If $\theta(v_i) \neq 0$ for some  $v_i \in V_{n-1-p}$ we may   alter our basis  in this degree so
that $\theta(v_j) = \delta_{ij}u$   for   $v_j \in V_{n-1-p}$ where $\delta_{ij}$ is the Kronecker delta function.
Then we see   $\theta$ is a   cycle if and only if  $c_kc_{ik}(v) = 0$  for all $v \in V.$
Translating, we have shown a derivation  $\theta \in \Der_{p}(\mathcal{L}_X, \mathcal{L}_{S^n}; \mathcal{L}_f)$ not vanishing on $V_{n-1-p}$
is a cycle if and only if there exists $x \in H^{\leq n-2}(X; \Q)$ such that
$xz \neq 0$ implies  $ H(f; \Q)(z) = 0$ for all $z \in H^n(X; \Q).$

Next let $\theta_a, \theta_b \in \Der(\mathcal{L}_X,\mathcal{L}_{S^n}; \mathcal{L}_f)$
be cycles   of degree $p$ and $q$, respectively.   As in the discussion  preceding \corref{cor:Wh
based}, let $\zeta^*_x = (0, \theta_x) \in
\Rel(\ad_{\mathcal{L}_f}),$ $x =a, b,$ be the corresponding cycles  so that, by (\ref{eq:Wh Relad*}),
the derivation cycle
$$ \lbr \theta_a, \theta_b \rbr = \induceddd \circ \{\Theta_a,
\Theta_b \} \circ \lambda
 \in \Der_{p+q-1}(\mathcal{L}_X, \mathcal{L}_{S^n}; \mathcal{L}_f)$$
represents  $[\langle \theta_a\rangle, \langle \theta_b \rangle ]_w.$
Using    (\ref{eq:quad}) again, (\ref{eq:Wh prod on v})  and the fact that
 $\induceddd(x)   = 0$ for $x = a,b$,  we obtain
\begin{align*}  \lbr \theta_a, \theta_b \rbr(v)  & = \pm \induceddd \circ \Theta_b
\circ \Theta_a \circ \lambda (d_X v) \\
 & =  \sum_{i<j}\pm c_{ij}(v)\left( [\theta_a(v_i),
\theta_b(v_j)] \pm [\theta_b(v_i),
\theta_a(v_j)] \right) \pm \sum_i 2c_{ii}(v)[\theta_a(v_i), \theta_b(v_i)].
\end{align*}
From this we conclude that $\lbr \theta_a, \theta_b \rbr$   nonvanishing for derivation cycles $\theta_a, \theta_b$  implies $\theta_a(v_i)\neq 0$ for some $v_i \in V_{n-1-p}$ and $\theta_b(v_j) \neq 0$  for some $v_j \in V_{n-1-q}$ and for this $i, j$ we have $c_{ij}(v) \neq 0$ for some $v \in V.$    By the above computation, the fact that $\theta_a$ and $\theta_b$ are cycles implies  $c_k c_{ik}(w) =  c_kc_{jk}(w) = 0$ for all $w \in V.$
Conversely, suppose there exist indices $i, j$ such that $c_{ij}(v) \neq 0$ and $c_k c_{ik}(w) = c_kc_{jk}(w) = 0$ for all $w \in V.$ Then   define $\theta_a, \theta_b$
by setting $\theta_a(v_l) = \delta_{li}u$ and $\theta_b(v_l) = \delta_{lj}u$ and extend  by the $\mathcal{L}_f$-derivation law.
  By the preceding paragraph,  the $\theta_a$ and $\theta_b$ are derivation cycles. Computing as   above $\lbr \theta_a, \theta_b \rbr(v) = \pm c_{ij}[u,u]$ is non-vanishing.   Combining and translating to cohomology, we have shown   that there exists a nontrivial pairing $\lbr \theta_a, \theta_b \rbr \in
\Der_{p+q-1}(\mathcal{L}_X, \mathcal{L}_{S^n}; \mathcal{L}_f)$ for cycles $\theta_a \in
\Der_{p}(\mathcal{L}_X, \mathcal{L}_{S^n}; \mathcal{L}_f)$ and
 $\theta_b  \in \Der_{ q}(\mathcal{L}_X, \mathcal{L}_{S^n}; \mathcal{L}_f)$ if and only
if there exists a pair $x, y \in H^{\leq n-2}(X; \Q)$  satisfying (i) and (ii).

Finally,  suppose  $\theta_a, \theta_b \in \Der(\mathcal{L}_X,\mathcal{L}_{S^n}; \mathcal{L}_f)$
are cycles   of degree $p$ and $q$ with $\lbr \theta_a, \theta_b \rbr \neq 0$.   As
above, let $v_i \in V_{n-1-p}$ and $v_j \in V_{n-1-q}$ be   basis elements
so that
$\theta_a(v_i) \neq 0$ and $\theta_b(v_j) \neq 0.$  If $p \neq q$ we  arrange our basis in degree $n-1-p$ and $n-1-q$ so that
 $\theta_a(v_l)  = \delta_{il}u$  and $\theta_b(v_m) =\delta_{jm}u$ for $v_l \in V_{n-1-p}$ and $v_m \in V_{n-1-q}.$
If $p= q$,  we  must allow for the case $v_i = v_j$.  In this case, we may arrange the basis in degree $n-1-p$
so that $\theta_a(v_l)  = c_a \delta_{il}u$  and $\theta_b(v_m) =c_b \delta_{jm}u$
for $v_l, v_m \in V_{n-1-p}.$ Here $c_a, c_b \neq 0$ and can be taken to be $1$ when $i \neq j$.  We  use this identification
in all cases by taking  $c_a= 1$ and $c_b = 1$ except, perhaps, when $i =j$.

 Let  $x = s(v_i) \in H^{n-p}(X; \Q)$   and $y = s(v_j) \in H^{n-q}(X; \Q)$ be the corresponding cohomology elements.
Then the pair $x, y$  satisfy (i) and (ii).  We show $\lbr \theta_a, \theta_b \rbr
\in \Der_{p+q-1}(\mathcal{L}_X, \mathcal{L}_{S^n}; \mathcal{L}_f)$ bounds if and only if the pair $x,y$ violates (iii).
 Since $ \lbr \theta_a, \theta_b \rbr$ is nonvanishing, the preceding discussion shows
there is  a vector $v \in V_{2n-1-p-q}$ with $$\lbr \theta_a, \theta_b \rbr(v) = \pm c_a c_b c_{ij}(v) [u, u] \neq 0.$$
Suppose
$\lbr \theta_a, \theta_b \rbr =   D_{\mathcal{L}_f}(\theta)$
for some $\theta \in \Der_{p+q}(\mathcal{L}_X, \mathcal{L}_{S^n}; \mathcal{L}_f).$
Applying this to $v \in V$ using (\ref{eq:quad}) we obtain
$$ \pm c_a c_b c_{ij}(v) [u, u] = \lbr \theta_a, \theta_b\rbr(v) =  D_{\mathcal{L}_f}(\theta)(v) = \pm\theta(d_X v)
= \pm \sum_{r < k}c_kc_{rk}(v)[\theta(v_r), u].$$
We conclude that if $\lbr \theta_a, \theta_b \rbr$ is a nonvanishing boundary then there is $v \in V$ with
$c_{ij}(v) \neq 0$ and $c_kc_{rk}(v) \neq 0$ for some $r$ which directly translates  to imply the pair $x,y$ violates
(iii) with $w= s(v_r)$ and $z = s(v_k)$.
Conversely, if $x,y$ violate (iii) then $c_k \neq 0$ and there exists some $v \in V_{2n-1-p-q}$ such that
  $c_{ij}(v) \neq 0$  and $c_{rk}(v) \neq 0$   for some index $r$.  Notice that  $v_r \in V_{n-1-p-q}.$
Define a derivation $\theta \in \Der_{p+q}(\mathcal{L}_X, \mathcal{L}_{S^n}; \mathcal{L}_f)$ by setting $\theta(v_l) = \delta_{lr}u$ and
 extending.   We then see $D_{\mathcal{L}_f}
(\theta)(v) = \pm c_k c_{rk}(v)[u, u] \neq 0$ while
$\lbr \theta_a, \theta_b \rbr (v) = \pm    c_a c_b c_{ij}(v)[u,u] \neq 0.$
To complete the proof, we show   the derivations $D_{\mathcal{L}_f}(\theta)$
and $\lbr \theta_a, \theta_b \rbr$  differ by a constant.  Note that both derivations increase bracket length. This implies they both vanish on $V$ except in degree $2n-1-p-q.$
In this degree, they are linear maps $V_{2n-1-p-q} \to \Q([u,u])$. We have shown both are  nonzero on a particular vector $v \in V_{2n-1-p-q}.$
Since the target is one-dimensional, there is a constant $c \neq 0$
such that
$\lbr \theta_a, \theta_b \rbr =  cD_{\mathcal{L}_f}(\theta) = D_{\mathcal{L}_f}(c\theta)$, as needed.
\end{proof}

We conclude with an example realizing the inequality
$$\WL_\Q(\map(X, Y; f)) > \WL_\Q(\map(X, Y; 0))  = \WL_\Q(Y)$$ for
some map $f \colon X \to Y.$
\begin{example} \label{ex:nonabelian}
Let $X = S^3$, and let $Y$ be a space with Sullivan minimal model
$\Lambda(x_1, x_2, x_3, y; d),$ the free DG  algebra with generators
of degrees $|x_1| = 2, |x_2| = |x_3| = 3,$ and $|y| = 7$.  Define
the (degree $+1$) differential here by setting $d(x_i) = 0$ and
$d(y_7) = x_1x_2x_3.$  Then  $Y$ has vanishing Whitehead products
since $d$ has no quadratic term \cite[Prop.13.16]{F-H-T}.
Consequently, $\WL_\Q(\map(X, Y;0)) = 1$ by \thmref{thm:WL null
component}. The Quillen minimal model $(\mathcal{L}_Y, d_Y)$ for $Y$
is of the form $\L(W; d_Y)$ where $W = s^{-1}\widetilde{H}_*(Y,
\Q)$. We use that the quadratic part of the differential $d_Y$ is
dual to the cup-product in $H^*(Y, \Q)$ (see
\cite[Sec.22(e)]{F-H-T}). In low degrees, we see $W$ contains
elements $w_1, w_2, w_3, w_{1,1}, w_{1,2}, w_{1,3}, w_{2,3}$ with
$|w_1| = 1, |w_2| = |w_3| = 2, |w_{1,1}| = 3, |w_{1,2}| = |w_{1, 3}|
= 4$ and $|w_{2,3}| = 5.$ Here $w_i$ corresponds to $x_i$ and
$w_{i,j}$ to the cup-product $x_i \cdot x_j$.  We may write the
differential as
$$ d_Y(w_1) = d_Y(w_2) = d_Y(w_3) = 0 \hbox{ \ \ with  \ \ }
  d_Y(w_{1,1}) = \frac{1}{2} [w_1, w_1],
$$ $$
 d_Y(w_{1, 2}) = [w_1, w_2], \,
d_Y(w_{1, 3}) = [w_1, w_3] \hbox{ \ \ and \ \ } d_Y(w_{2, 3}) =
[w_2, w_3]$$ on these generators.
 Write
the Quillen minimal model for $S^3$ as $\L(v; 0)$ with $v$ in degree
$2$ and let $f \colon S^3 \to Y$ correspond, after rationalization,
to the class $w_3 \in H_2(\mathcal{L}_Y).$ That is,
$\mathcal{L}_f(v) = w_3.$ We show that $\WL_\Q(\map(S^3, Y; f)) \geq
2 .$  Observe that an element $\zeta_a = (\chi_a, \theta_a) \in
\Rel_p(\ad_{\mathcal{L}_f})$ is a
$\delta_{\ad_{\mathcal{L}_f}}$-cycle if $d_Y(\chi_a) = 0$ and
$d_Y(\theta_a(v)) = -[\chi_a, w_3].$ Thus $\zeta_a = (w_1,
\theta_a)$ and $\zeta_b = (w_2, \theta_b)$ are
$\delta_{\ad_{\mathcal{L}_f}}$-cycles of degree $2$ and $3$
respectively where $\theta_a(v) = -w_{1,3}$ and $\theta_b(v) =
-w_{2,3}.$ Write $\alpha \in \pi_2(\map(X, Y;f)) \otimes \Q$ and
$\beta \in \pi_3(\map(X, Y;f)) \otimes \Q$ for the corresponding
homotopy elements as in \secref{sec:formula}. Applying
\thmref{thm:map}, their Whitehead product $[\alpha, \beta]_w \in
\pi_4(\map(X, Y;f)) \otimes \Q$ corresponds to the class represented
by the $\delta_{\ad_{\mathcal{L}_f}}$-cycle $\lbr \zeta_a, \zeta_b
\rbr = ([w_1, w_2], \{ \zeta_a, \zeta_b \}) \in
\Rel_4(\ad_{\mathcal{L}_f})$; where, by using formula (\ref{eq:Wh
prod on v}), we have that $\{ \zeta_a, \zeta_b \}(v) = - [w_2,
w_{1,3}] - [w_1, w_{2, 3}].$   This cannot be a boundary.  For if
 $\delta_{\ad_{\mathcal{L}_f}}(\eta, \theta)
= ([w_1, w_2], \{ \zeta_a, \zeta_b \})$, then we have $\eta = -
w_{1,2} + \chi$ for some $\chi$ a cycle in $\mathcal{L}_Y$ of degree
$4$. Since $Y$ has no rational homotopy of degree $5$ (direct from
the Sullivan model), we see that $\chi = d_Y(\xi)$ for some $\xi \in
\mathcal{L}_Y$. Further, we then obtain that
$$\ad_{\mathcal{L}_f}(- w_{1,2} + d_Y(\xi)) (v) +
D_{\mathcal{L}_f}(\theta)(v) = \{ \zeta_a, \zeta_b \}(v),$$ which
implies that
$$d_Y([\xi, w_3] + \theta(v)) = [w_{1,2}, w_3]  - [w_1, w_{2,3}] - [w_2, w_{1,3}].$$
However, when $d_Y$ is applied to this latter term, it yields $2 [
[w_1, w_2], w_3]$ and not zero, so it cannot be a cycle (boundary).
 We
conclude that $[\alpha, \beta]_w \neq 0.$
\end{example}



\providecommand{\bysame}{\leavevmode\hbox
to3em{\hrulefill}\thinspace}
\providecommand{\MR}{\relax\ifhmode\unskip\space\fi MR }
\providecommand{\MRhref}[2]{%
  \href{http://www.ams.org/mathscinet-getitem?mr=#1}{#2}
} \providecommand{\href}[2]{#2}

\end{document}